\documentclass{amsart}

\usepackage{amssymb}
\usepackage{amsmath,amsfonts,amsthm,amstext}
\usepackage{amsmath}
\usepackage{mathrsfs}  
\usepackage[all]{xy}
\usepackage{xcolor}
\usepackage{tikz-cd}

\newtheorem{theorem}{Theorem}[section]
\newtheorem{lemma}[theorem]{Lemma}

\newtheorem{prop}[theorem]{Proposition}
\newtheorem{cor}[theorem]{Corollary}

\newcommand{\Z}{\mathbb{Z}}
\newcommand{\z}{\mathbb{Z}}

\newcommand{\X}{{\mathcal X}}
\theoremstyle{definition}

\begin{document}

		\title[]{
		On the residual finiteness of the non-abelian tensor square $G \otimes G$, the non-abelian exterior square $G \wedge G$ and the weak commutativity construction $\X(G)$} \author{ 
 Ayrton Anjos Teixeira, Lucas Barroso Rocha, Dessislava H. Kochloukova}
	\address{State University of Campinas (UNICAMP), SP, Brazil \\
		\newline
		email : desi@ime.unicamp.br
		} 
	\email{}
	%\subjclass[2010]{Primary 20J05; Secondary 20F05;}
	\date{}

\begin{abstract}   We prove that if $G$ is a residually finite group,  $\widehat{G}$ is its profinite completion and the  map $H_2(G, \mathbb{Z}) $ $ \to H_2(\widehat{G}, \widehat{\mathbb{Z}}),$ induced by the canonical map  $G \to \widehat{G}$, is injective, then the non-abelian exterior square $G \wedge G$ is residually finite. 
We show that if $G$ is a finitely presented centre-by-metabelian group then  $\nu(G)$, the non-abelian tensor product  $G \otimes G$ and the non-abelian tensor square $G \wedge G$  are  residually finite.  Furthermore we prove that if $G$ be a finitely presented metabelian group then the weak commutative construction $\X(G)$ is residually finite.  We discuss a criterion for the $q$-exterior  square $G \wedge^q G$ to be residually finite. 
\end{abstract}

\maketitle

\section{Introduction}
 In this paper we study some sufficient conditions that force the non-abelian tensor product $G \otimes G$, the non-abelian exterior product $G \wedge G$  and the related groups $\nu(G)$ and $\mathcal{X}(G)$ to be residually finite. One of the approaches we adopt uses homological  methods. Homological methods were quite often used in the early days of the study of the non-abelian tensor product $G \otimes G$, defined by  Dennis \cite{D} and developed by Brown and Loday \cite{B-L1,B-L2}. 
 The non-abelian tensor product $G \otimes G$ preserves many properties of the commutator subgroup of the original group $G$. For example Brown, Johnson and Robertson showed in
 \cite{B-J-R}  that $G\otimes G$  is $\pi$-finite, nilpotent, soluble, if $G$ is  $\pi$-finite, nilpotent, soluble respectively and Moravec proved in \cite{M} that whenever $G$ is polycyclic $G \otimes G$ is polycyclic too.  Using homological methods Ellis showed in \cite{Ellis}  that if $G$ is finite then $G \otimes G$ is finite.

The non-abelian tensor product is strongly linked with the groups $\nu(G)$ and $\mathcal{X}(G)$.
Let $G$ be a group and ${G}^{\varphi}$ be an isomorphic copy of $G$. We write ${g}^{\varphi}$ for the image of $g \in G$ in ${G}^{\varphi}$. The group $\X(G)$ was defined by Sidki in \cite{Said} by the following presentation
$$\X(G) = \langle G, {G}^{\varphi} | \ [g, {g}^{\varphi}] = 1 \hbox{ for } g \in G \rangle.$$
In \cite{Said} Sidki showed that if $G$ is $\pi$-finite ($\pi$ is a set of primes), perfect, soluble, finite nilpotent, then $\X(G)$ is resp.  $\pi$-finite, perfect, soluble, finite nilpotent  too. Later similar result was proved for finitely generated  nilpotent groups \cite{G-R-S} and for  polycyclic-by-finite groups  \cite{LimaOliveira}. In \cite{B-K} Bridson and Kochloukova showed that if $G$ is finitely presented then $\X(G)$ is finitely presented.

The group $\nu(G)$ was defined by Rocco in \cite{Norai} by the following presentation
$$\nu(G) = \langle G, {G}^{\varphi} \ | \ [g_1, {g}_2^{\varphi}]^{g_3} = [g_1, {g}_2^{\varphi}]^{{g}_3^{\varphi}} = [g_1^{g_3}, (g_2^{g_3})^{\varphi}] \hbox{ for } g_1, g_2, g_3 \in G \rangle.$$
The subgroup $[G, {G}^{\varphi}] \leq \nu(G)$ is isomorphic to the non-abelian tensor product $G \otimes G$.

For a group $G$ we denote by $ \widehat{G}$ the profinite completion of $G$.

 \medskip
{\bf Theorem A} {\it Let $G$ be a discrete group  and $p$ be a prime integer.

   a1)  Suppose that $G$ is residually finite and the  map $$H_2(G, \mathbb{Z}) \to H_2(\widehat{G}, \widehat{\mathbb{Z}}),$$ induced by the canonical map  $G \to \widehat{G}$, is injective. Then $G \wedge G$ is residually finite.
 
 a2)  Suppose that $G$ is residually finite and the  map $H_2(G, \Z) \to \varprojlim_U H_2(G/U, \Z)$, induced by the canonical projections $G \to G/ U$ (for $ \widehat{G} \simeq \varprojlim G/U$, $G/ U$ finite groups) is injective. Then $G \wedge G$ is residually finite.

 b) 
  Suppose that $G$ is a finitely generated residually finite group such that $G \wedge G$ is finitely generated. Then $G \wedge G$ is residually finite if and only if the map  $$H_2(G, \mathbb{Z}) \to H_2(\widehat{G}, \widehat{\mathbb{Z}}),$$ induced by the canonical map  $G \to \widehat{G}$, is injective.
 
     c1) Suppose that $G$ is residually $p$-finite and the map $H_2(G, \Z) \to \varprojlim_U H_2(G/U, \Z)$, induced by the canonical projections $G \to G/ U$ (for $ \widehat{G}_p \simeq \varprojlim G/U$, $G/ U$ finite $p$-groups) is injective then $G \wedge G$ is residually $p$-finite.
 
 c2)  Suppose that $G$ is residually $p$-finite and    the canonical map $H_2(G, \mathbb{Z}) \to H_2(\widehat{G}_p, {\mathbb{Z}}_p)$, induced by the canonical map  $G \to \widehat{G}_p$, is injective then $G \wedge G$ is residually $p$-finite. 
 
 d) If $G/ G'$ is finitely generated and either has no 2-torsion or  $G'$ has a complement in $G$ then $G \otimes G$ is residually finite if and only if $G \wedge G$ is residually finite.
 
 e)   If $G/ G'$ is finitely generated and either has no 2-torsion or  $G'$ has a complement in $G$ then $G \otimes G$ is residually $p$-finite if and only if $G \wedge G$ is residually $p$-finite and the torsion subgroup of $\Delta(G)$ is $p$-finite. In particular,  if $G/ G'$ is finitely generated and either has no 2-torsion or  $G'$ has a complement in $G$, then   $G \otimes G$ is residually $p$-finite for every prime $p$ if and only if $\Delta(G)$ is torsion-free and $G \wedge G$ is residually $p$-finite  for every prime $p$.
 }

\medskip
 
 The notion of (cohomologically) good groups was developed by Serre in \cite{Serre}. A discrete group $G$ is called good if for the profinite completion $\widehat{G}$ of $G$ we have that the canonical map $G \to \widehat{G}$ induces an isomorphism $$H^i(\widehat{G}, A) \to H^i(G, A)$$ for every finite discrete $\widehat{G}$-module $A$ and every $i \geq 0$. 
 A group $G $ belongs to the class $\mathcal{A}_n$ if the above holds for all $i \leq n$. 
 
 We define the class  $\mathcal{A}_n^{hom}$ of groups $G$ such that the map $$ H_i(G, A) \to H_i(\widehat{G}, A)$$ is an isomorphism for $i \leq n$. A group $G$ is homologically good if  $G \in \mathcal{A}_n^{hom}$ for all $n$.  It turns out that the classes $\mathcal{A}_n$ and $\mathcal{A}_n^{hom}$ coincide, see \cite{BR-K}.

\medskip

{\bf Corollary B} {\it If  $G \in \mathcal{A}^{hom}_2$, $G$ is residually finite and $H_2(G, \z)$ is finitely generated then the map $H_2(G, \Z) \to H_2( \widehat{G}, \widehat{\Z})$ is injective and $G \wedge G$ is residually finite.}

\medskip 
In sequels \cite{AT-K},  \cite{BR-K} of this paper   we  explore consequences of Theorem A and Corollary B  and the fact that $ \mathcal{A}_2=  \mathcal{A}^{hom}_2$.

 By the results of Roseblade and Jategaonkar in \cite{J}, \cite{R} all  finitely generated abelian-by-polycyclic groups are residually finite.
In \cite{Groves}  Groves proved  that finitely presented centre-by-metabelian groups are abelian-by-polycyclic and hence such  groups are residually finite.  But when the condition of finite presentability is removed the centre-by-metabelian group is not necessary residually finite \cite{Groves2}. The methods used in the proof of the following results are homology free and  use the Bieri-Strebel invariant defined in \cite{B-S}.

\medskip
{\bf Theorem C} 
 {\it 
 Let $G$ be a finitely presented centre-by-metabelian group. Then $\nu(G)$ is a finitely generated,  abelian-by-polycyclic group with the abelian normal subgroup being characteristic. In particular  $\nu(G)$,  $G \otimes G$, $G \wedge G$  are  residually finite. }

\medskip
{\bf Theorem D} {\it 
 Let $G$ be a finitely presented metabelian group. Then $\X(G)$ is residually finite.}
 
 \medskip
 In the last section we show that the methods  used to show residual finiteness of  $G \wedge G$ could be adapted for the $q$-exterior square $G \wedge^q G$ defined by Conduche and Rodr\'iguez-Fern\'andez in \cite{C-R}.  In \cite{Ellis2} a generalisation of $\nu(G)$  related to $G \otimes^q G$ and $G \wedge^q G$ was defined. It was later denoted by $\nu^q(G)$ in \cite{Bu-Ro} and further properties of $\nu^q(G)$ were investigated in \cite{Bu-Ro}. 
 The group $\nu^q(G)$ contains $G \otimes^q G$ as a subgroup and $G \wedge^q G$ is a quotient of $G \otimes^q G$. 
 Since $G \wedge^1 G \simeq G$ we consider $G \wedge^q G$ only for $q \geq 2$.
 
 \medskip
 {\bf Proposition E} {\it Let $q \geq 2$ be an integer  and $p$ be a prime integer.
  Suppose that $G$ is a finitely generated residually finite group. Then 
  
  a) $G \wedge^q G$ is residually finite if and only if the map $H_2(G, \mathbb{F}_q) \to H_2(\widehat{G}, \mathbb{F}_q)$, induced by the canonical map  $G \to \widehat{G}$, is injective.
  
  b)  $ G \wedge^q G$ is residually $p$-finite if and only if  the map $H_2(G, \mathbb{F}_q) \to H_2(\widehat{G}_p, \mathbb{F}_q) $, induced by the canonical map  $G \to \widehat{G}$, is injective.}
\medskip

{\bf Theorem F} {\it Let $ q \geq 1$ be an integer.
Suppose that  $G$ is finitely presented, central-by-metabelian group. Then there is an abelian characteristic subgroup $C$ of $\nu^q(G)$  such that $\nu^q(G)/ C$ is polycyclic.   In particular, $\nu^q(G)$, $G \otimes^q G$ and $G \wedge^q G$ are residually finite.} 

\medskip
Theorem A is proved in section  \ref{sectA},  Corollary B is considered in section \ref{goodness}. In section  \ref{sec-final} we prove Theorem C and Theorem D is proved in section \ref{meta}. And in the final section \ref{q-tensor123} we prove Proposition E and Theorem F.

\medskip
{\bf Acknowledgements}    This study was financed in part by the Coordenação de Aperfeiçoamento de Pessoal de Nível Superior - Brasil (CAPES) - Finance Code 001 (Barroso Rocha is supported by a PhD Capes grant), Anjos Teixeira is supported by a PhD grant process No. 168757/2023-0, Conselho Nacional de Desenvolvimento Científico e Tecnológico - CNPq and      Kochloukova is partially supported by bolsa de produtividade em pesquisa CNPq 301222/2026-6 and FAPESP 2024/14914-9.

\section{Preliminaries}

 The following result of Philip Hall implies that finitely generated metabelian groups are residually finite.
 
 \begin{theorem}
  \cite{Hall2} Every finitely generated abelian-by-nilpotent group is residually finite.
  \end{theorem}
  
  The above was later generalised by Jategaonkar and by Roseblade.
  
  \begin{theorem} \cite{J}, \cite{R}   Every finitely generated abelian-by-polycyclic group is residually finite.
  \end{theorem}

It is worth noting that there is a finitely generated soluble non-metabelian group that is not residually finite, see \cite{Hall}.

\begin{theorem} \cite{Segal} \label{segal}
Let $\Gamma$ be a finitely generated group with an abelian normal subgroup
$A$ such that $\Gamma /A$ is polycyclic.

(i) If $A$ is a $p$-group for some prime $p$, then $\Gamma$ has a normal subgroup of finite index which is residually a finite $p$-group.

(ii) If $A$ is torsion-free, then for almost all primes $p$, $\Gamma$ has a normal subgroup of finite index which is residually a finite $p$-group
\end{theorem}
 
 \iffalse{By \cite{H} polycyclic groups are residually finite.}\fi
 
 \subsection{Structure of $G \otimes G$ and $\nu(G)$} \label{intro-non-abelian} 
 
 We identify $G \otimes G$ with the normal subgroup $[G, G^{\varphi}]$ of $\nu(G)$.
 
 We write $W_0(G)$ for the kernel of the epimorphism  $G \otimes G \to G'$ that sends $g_1 \otimes g_2$ to $[g_1, g_2]$. We write $\Delta(G)$ for the subgroup of $G \otimes G$ generated by $\{ g \otimes g ~ | ~ g \in G \}$.

 Recall the definition of the non-abelian exterior  square $G \wedge G$ as the quotient  $G \otimes G/ \Delta(G)$. The commutator map  $G \wedge G \to G'$ induces a map  $( G \wedge G) ^{ ab} \to (G')^{ ab}$. By definition $$D_0(G) = [G, {G}^{\varphi}],   \ L_0(G) = \langle {g}^{\varphi} g^{ -1} \ | \ g \in G \rangle \hbox{ and }W_0(G) = D_0(G) \cap L_0(G)$$ are normal subgroups of $\nu(G)$. For simplicity we write $L_0,D_0$ and $W_0$ for $L_0(G),$$D_0(G)$ and $W_0(G)$ when it is clear from the context which is the group $G$.

 Recall that 
$$\X(G) = \langle G, {G}^{\varphi} | \ [g, {g}^{\varphi}] = 1 \hbox{ for } g \in G \rangle$$ 
 There is a homomorphism $$\rho : \mathcal{X} (G) \to G \times G \times G$$
 with kernel $W(G) = L(G) \cap D(G)$, where  $$D(G) = [G, {G}^{\varphi}],  \ L(G) = \langle {g}^{\varphi} g^{ -1} \ | \ g \in G \rangle$$ are normal subgroups of $\X(G)$. By definition $\rho(g) = (g,g,1)$ and $\rho(g^{\varphi}) = ( 1,g,g)$.

Similarly there is a homomorphism  $$\rho_0 : \nu(G) \to G \times G \times G$$
 with kernel $W_0(G)$ defined by $\rho_0(g) = (g,g,1)$ and $\rho_0(g^{\varphi}) = ( 1,g,g)$. Thus $Im (\rho) = Im (\rho_0)$.

 By   \cite[Lemma 2.2]{Norai2}  for $R(G) = [G, L(G), G^{\varphi}]$ we have $$W(G)/ R(G) \simeq H_2(G, \mathbb{Z})$$ 
As pointed out in \cite[p.68-69]{Norai} $$\nu(G)/ \Delta(G) \simeq \mathcal{X} (G)/ R(G)$$ with $W_0(G)/ \Delta(G) \simeq W(G)/ R(G)$, hence
$$W_0(G)/ \Delta(G) \simeq H_2(G, \mathbb{Z}).$$
 
 By \cite{B-L1}, \cite{B-L2} $\Delta(G)$ is a quotient of $\Gamma(G/ G')$, where $\Gamma$ denotes the Whitehead quadratic functor and $$G \wedge G = (G \otimes G) / \Delta(G)$$ maps surjectively to $G'$ with kernel isomorphic to the Schur multiplier $M(G) = H_2(G, \mathbb{Z})$. In particular since $\Gamma(G/ G')$ is finitely generated if $G/ G'$ is finitely generated, we conclude that $\Delta(G)$ is finitely generated. Furthermore as pointed out in \cite[p.68-69]{Norai} $\Delta(G)$ is central in $\nu(G)$.
 
We quote some  general results on $G \wedge G$ and $G \otimes G$.
 
 \begin{theorem} \cite{B-F-M} \label{wedge}  Let $G$ be a group and $F$ be a free group such that $G \simeq F/ R$ for some normal subgroup $R$ of $F$. Then $G \wedge G \simeq F'/ [F, R]$.
 \end{theorem} 
 
 \begin{theorem} \cite{B-F-M} \label{2tor}
 Let $G$ be a group such that $G/ G'$ is finitely generated. If $G/ G'$ has no
elements of order 2, or if $G'$ has a complement in $G$, then $$G \otimes G \simeq \Delta(G) \times (G \wedge G).$$
 \end{theorem}
 
 {\bf Remark} The proof of the above theorem shows that under the conditions of Theorem \ref{2tor} $\Delta(G) \simeq \Delta(G/ G') \simeq \Gamma(G/ G')$.
 
 \medskip
 We observe that in \cite{B-F-M} the authors use the notation $\nabla(G)$ for $\Delta(G)$.
 The condition on the absense of 2-torsion in $G/ G'$  in the statement of Theorem \ref{2tor} comes from the fact that in this case the map $\Gamma(G/G') \to \Delta(G/ G')$ is injective, where $\Gamma(G/ G')$ is the Whitehead quadratic functor, and this implies an isomorphism $\Delta(G/ G') \simeq \Delta(G)$. For more details on the Whitehead quadratic functor we refer the reader to section 1 from \cite{M-M-O}.
 
\subsection{The Bieri-Strebel criterion for finitely presented metabelian groups}
  
 We recall first some properties of the first $\Sigma$-invariant as defined in \cite{B-S}.
 Let $Q$ be a finitely generated abelian group. By definition the character sphere $$S(Q) = Hom(Q, \mathbb{R}) \setminus \{ 0 \}/ \sim$$ where for $\chi_1, \chi_2 \in Hom(Q, \mathbb{R}) \setminus \{ 0 \}$ we have $\chi_1 \sim \chi_2$ if and only if there is a positive real number $r$ such that $\chi_1 = r \chi_2$. 
 
 If $Q \simeq \mathbb{Z}^n \oplus F$, where $F$ is finite, we can identify $S(Q)$ with the unit sphere $S^{n-1}$. The equivalence class $\mathbb{R}_{>0} \chi$ is denoted by $[\chi]$.
 
 Let $M$ be a finitely generated $\mathbb{Z} Q$-module. By definition
 $$
 \Sigma_M(Q) = \{ [\chi] \ | \ M \hbox{ is finitely generated as } \mathbb{Z} Q_{\chi}-\hbox{module} \}
 $$
  where $Q_{\chi} = \{ q \in Q \ | \ \chi(q) \geq 0 \}$.
 The complement $S(Q) \setminus \Sigma_M(Q)$ is denoted by $\Sigma_M^c(Q)$.  
 
 \begin{lemma} \cite[Prop. 2.2]{B-S} \label{ttt}
a) Let $ 0 \to M_1 \to M \to M_2 \to 0$ be a short exact sequence of finitely generated  $\mathbb{Z} Q$-modules. Then
 $$\Sigma_M^c(Q) = \Sigma_{M_1}^c(Q) \cup \Sigma_{M_2}^c (Q).$$
 
 b) $\Sigma_M(Q) = \Sigma_{\mathbb{Z} Q / I} (Q)$, where $I$ is the annihilator ideal of $M$.
 \end{lemma}
  
  The classification of the finitely presented  metabelian groups depends on the $\Sigma$-invariant.
  
  \begin{theorem} \cite[Thm. A,ii)]{B-S} \label{classification} 
  Let $ 1 \to M \to G \to Q \to 1$ be a short exact sequence of groups, where $G$ is finitely generated and both $M$ and $Q$ are abelian. Then $G$ is finitely presented if and only if $\Sigma_M^c(Q)$ does not have antipodal points.
  \end{theorem}

\subsection{Homology and cohomology} If not stated otherwise a group $G$ is a discrete (abstract) group. We denote by $ \widehat{G}$ its profinite completion, hence $\widehat{G} \simeq   \varprojlim_U G / U$ where the inverse limit is over all normal subgroups $U$ of  finite index in $G$, and by $ \widehat{G}_p$ its pro-$p$ completion, where $p$ is a prime integer,  hence $\widehat{G}_p \simeq   \varprojlim_V G / U$ where the inverse limit is over all normal subgroups $V$ of  $p$-power  index in $G$.
For example the profinite completion of $\mathbb{Z}$ is $ \widehat{\mathbb{Z}}$ and the pro-$p$ completion is the ring of $p$-adic integers  $\widehat{\mathbb{Z}}_p$.

 We use the standard notation $H_n(G, - )$ and $H^n(G, - )$ for homology and cohomology, for basic properties the reader is refered to \cite{Rotman}. We write $H_n(\widehat{G}, - )$, $H_n(\widehat{G}_p, - )$, $H^n(\widehat{G}, -)$, $H^n(\widehat{G}_p, -)$ for profinite homology and cohomology,  for basic properties the reader is refered to \cite{R-Z}, \cite{Serre}.  For $
A =  \varprojlim_i A_i$, each $A_i$ is a finite $G/ U_i$-module, hence a  finite $G/ U$-module for all $U \leq U_i$  we have 
  $$H_n(\widehat{G}, A) \simeq  \varprojlim_{U,i}H_n(G/ U, A_i)$$ Dually for $B = \varinjlim_i B_i$, each $B_i$ is a finite $G/ U_i$-module, hence a finite $G/ U$-module for all $U \leq U_i$, we have $$H^n(\widehat{G}, B) \simeq  \varinjlim_{U,i}H^n(G/ U, B_i)$$

  \subsection{Cohomologically good groups}     Following Serre  \cite{Serre} and Lorensen \cite{L} a discrete group $G \in \mathcal{A}_n$ if for the profinite completion $\widehat{G}$ of $G$ the canonical map $G \to \widehat{G}$ induces an isomorphism $$H^i(\widehat{G}, A) \to H^i(G, A)$$ for every finite discrete $\widehat{G}$-module $A$ and every $ i \leq n$.  A group is called (cohomologically) good if $G \in \mathcal{A}_n$ for all $n$.
    
       The importance of the class $\mathcal{A}_2$ lies in the following result.
  
  \begin{prop}  \cite{L}, \cite{Serre} \label{SerSer}  If
 $G_0 \in \mathcal{A}_2$ and $1 \to N_0 \to E_0 \to G_0 \to 1$ is a short exact sequence of groups with $N_0$ finitely generated, both $N_0$ and $G_0$ residually finite, then $E_0$ is residually finite.
    \end{prop}
    
     The homological version of the class $\mathcal{A}_n$ is the class of groups $\mathcal{A}_n^{hom}$ defined in the introduction.
    The fact that $\mathcal{A}_n = \mathcal{A}_n^{hom}$ is proved in \cite{BR-K}, generalising earlier result of Hillman and Kochloukova in \cite{H-K}. We do not need this result in the current paper  but we will consider applications of it mainly in  \cite{BR-K} and partially in  \cite{AT-K}.

      \iffalse{\begin{theorem} \cite{BR-K}  \label{surprise} Let $G$ be a discrete group,  $\mathcal{T}$ be a set of normal subgroups of finite index in $G$ such that $\mathcal{T}$ defines the profinite topology of $G$ and for every prime integer $p$ consider $\mathbb{F}_p$ as the trivial $G$-module. Then the following are equivalent:

a) $G \in {\mathcal{A}}_n$;

b) $G \in {\mathcal{A}}_n^{hom}$;

c)  the map $H^i ( \widehat{U}, \mathbb{F}_p) \to H^i(U, \mathbb{F}_p)$ is an isomorphism for all $ i \leq n$, $U \in \mathcal{T}$ and for every prime integer $p$;

d)  the map $H_i(U, \mathbb{F}_p) \to H_i ( \widehat{U}, \mathbb{F}_p)$ is an isomorphism for all $ i \leq n$,  $U \in \mathcal{T}$ and  for every prime integer $p$.
\end{theorem}
}\fi

\iffalse{A discrete group $G$ is {\it (cohomologically) $p$-good in dimension $\leq n$} if for every finite pro-$p$ $\mathbb{Z}_p[[\widehat{G}_p]]$-module $A$ the canonical map $G \to \widehat{G}_p$ induces na isomorphism 
$$H^{i}(\widehat{G}_p, A) \to H^i(G, A) \hbox{ for every } i \leq n.$$
A discrete group $G$ is {\it homologically $p$-good in dimension $\leq n$} if 
for every finite pro-$p$ $\mathbb{Z}_p[[\widehat{G}_p]]$-module $A$ the canonical map $G \to \widehat{G}_p$ induces an isomorphism 
$$ H_i(G, A) \to H_{i}(\widehat{G}_p, A) \hbox{ for every } i \leq n.$$}\fi

  \section{Proof of Theorem A} \label{sectA}
 
 This section contains various results that together prove Theorem A.
 
  In the following lemma we prove that a Hopf type formula is valid for profinite groups.
 For subsetsets $A,B$ of a profinite group  $F_0$ write $\overline{[A, B]}$ for the profinite subgroup of $F_0$ generated (topologically)  by $\{ [f_1, f_2] \ | \ f_1 \in A, f_2 \in B \}$. 
 \begin{lemma}
 Let $C$ be a profinite group and $F_0$ be a profinite group with a profinite  normal  subgroup $R_0$ such that $C \simeq F_0 / R_0$  and $H_2(F_0, \widehat{\mathbb{Z}}) = 0$. Then there is an isomorphism
 $$H_2(C, \widehat{\mathbb{Z}}) \simeq R_0 \cap \overline{[F_0, F_0]}/ \overline{[R_0, F_0]}. $$
 \end{lemma}
 
 \begin{proof}
 Consider the short exact sequence of profinite groups $$1 \to R_0 \to F_0 \to C \to 1$$
 and the associated Lyndon-Hoschild-Serre spectral sequence $$E_{i,j}^2 = H_i(C, H_j(R_0, \widehat{\mathbb{Z}}))$$ that converges to the profinite group $H_{i+ j}(F_0, \widehat{\mathbb{Z}})$. This spectral sequence gives rise a 5-term exact sequence in homology
 $$
  0 =  H_2(F_0, \widehat{\mathbb{Z}}) \to H_2(C, \widehat{\mathbb{Z}}) \to R_0/\overline{[R_0, F]} \to H_1(F_0, \widehat{\mathbb{Z}}) \to  H_1(C, \widehat{\mathbb{Z}}) \to 1$$
 Since $H_1(F_0, \widehat{\mathbb{Z}})  \simeq F_0/ \overline{[F_0, F_0]}$ we obtain that
 $$ H_2(C, \widehat{\mathbb{Z}}) \simeq Ker( R_0/\overline{[R_0, F]} \to H_1(F_0, \widehat{\mathbb{Z}})) \simeq$$ $$
 Ker ( R_0/ \overline{[R_0, F]} \to  F_0/ \overline{[F_0, F_0]}) = R_0 \cap \overline{[F_0, F_0]}/ \overline{[R_0, F_0]}.$$
 \end{proof}
 
 Let $G$ be a discrete group.
 Recall that $\widehat{G}$ is the profinite completion of $G$, hence $\widehat{G} \simeq   \varprojlim_U G / U$, where the inverse limit is over all normal subgroups $U$ of  finite index in $G$.
 We define the profinite non-abelian exterior product   $\widehat{G} \widehat{\wedge} \widehat{G}$ as the inverse limit   $ \varprojlim_U (G / U) \wedge (G/ U) $,  where the inverse limit is over all normal subgroups $U$ of  finite index in $G$.
 
 \begin{lemma} \label{res-finite}Suppose that $G$ is a discrete group such that $G$ is residually finite.

  a) If the map $H_2(G, \Z) \to \varprojlim_U H_2(G/U, \Z)$, induced by the canonical projections $G \to G/ U$ (for $ \widehat{G} \simeq \varprojlim G/U$, $G/ U$ finite groups) is injective then $G \wedge G$ is residually finite.
 
 b) If the canonical map $H_2(G, \mathbb{Z}) \to H_2(\widehat{G}, \widehat{\mathbb{Z}})$ is injective then $G \wedge G$ is residually finite.
 \end{lemma}
 
 \begin{proof}
  Consider the short exact sequence
  $$   1 \to H_2(G/U, \mathbb{Z}) \to (G / U) \wedge (G/ U) \to [G/U, G/U] \to 1 $$
  Then take
   inverse limit over all $G/ U$ finite and 
   note that  though the inverse limit is only a left exact (covariant) functor since  all groups  $H_2(G/U, \mathbb{Z})$ are  finite the inverse system $\{ H_2(G/U, \Z) \}_U$ satisfies the  Mittag-Leffler condition, hence  $\varprojlim_U ^1 H_2(G/U, \mathbb{Z}) = 0$ and   we  obtain an exact sequence
  $$ 1 \to \varprojlim_U  H_2(G/U, \mathbb{Z}) \to \varprojlim_U (G / U) \wedge (G/ U) \to \varprojlim_U [G/U, G/U] \to 1 $$
  
 Note that  $\widehat{G} \widehat{\wedge} \widehat{G} =  \varprojlim_U (G / U) \wedge (G/ U) $ is a profinite group and $  \varprojlim_U [G/U, G/U]$ is the commutator subgroup $\overline{[\widehat{G}, \widehat{G}]}$ of $ \widehat{G}$. \iffalse{The above sequence is exact on the right, since the maps in the inverse system $\{ [G/U, G/U] \}_U$ are surjective.}\fi
  
  The canonical map $G \to \widehat{G}$ induces a map $\beta : G \wedge G \to \widehat{G} \widehat{\wedge} \widehat{G}$  that induces  a diagram with exact rows
 
\[ 
\begin{tikzcd}
{ 1}\arrow{r}{}&  H_2(G, \mathbb{Z})  \arrow{r}
\arrow{d}{\alpha}
 &   G \wedge G \arrow{r}{} \arrow{d}{\beta}& \ [G,G] \arrow{r}{}\arrow{d}{\gamma}& { 1}\\%
{ 1} \arrow{r}{}& \varprojlim_U H_2({ G/U}, \mathbb{Z}) \arrow{r}{}  &  \widehat{G}  \ \widehat{\wedge} \  \widehat{ G} \arrow{r}{} & \ \overline{[\widehat{G},\widehat{G}]}  \arrow{r}{} & { 1}\\%
\end{tikzcd}
\]

   The maps $\Z \to \Z/ m \Z$ induce  a map $ \varprojlim_U H_2(G/U, \mathbb{Z}) \to  \varprojlim_{U,m} H_2(G/U, \mathbb{Z}/m \Z)$
  that combined with the isomorphism $\varprojlim_{U,m} H_2(G/U, \mathbb{Z}/m \Z) \simeq H_2(\widehat{G}, \widehat{\mathbb{Z}})$  give a map
  $$\mu : \varprojlim_U H_2(G/U, \mathbb{Z})  \to H_2(\widehat{G}, \widehat{\mathbb{Z}})$$
  
 In a) is given that $\alpha$ is injective and in b) is given that the map  $\mu \circ \alpha : H_2(G, \mathbb{Z}) \to  H_2(\widehat{G}, \widehat{\mathbb{Z}})$ is injective. Then in both cases  $\alpha$ is injective.
  
  Since $G$ is residually finite the canonical map $G \to \widehat{G}$ is injective, hence $\gamma$ is injetive. This implies, via the commutative diagram, that $\beta$ is injective but since $ \widehat{G} \widehat{\wedge} \widehat{G}$ is profinite, $ \widehat{G}   \widehat{\wedge}  \widehat{G}$  is residually finite, hence  $G \wedge G$ is residually finite.
 
 \medskip
 2) We give a second proof of b) in the case when $G$ is finitely generated. Actually an adaptation for not finitely generated group could be done but as extra care is needed to deal with non-finitely generated  free profinite groups and we already have one valid proof we restrict to the finitely generated case. Let $F$ be a finitely generated  free group such that $G \simeq F / R$. Then the profinite completion $ \widehat{F}$ is a finitely generated free profinite group and we consider the canonical map $F \to \widehat{F}$ as the inclusion map.  For a subset $S$ in $ \widehat{F}$ denote by $ \overline{S}$ the closure of $S$. For subsets $S_1, S_2$ of $ \widehat{F}$ denote by $\overline{[S_1, S_2]}$ the closure of the abstract subgroup generated by $\{ s_1^{-1} s_2^{-1} s_1 s_2 \ | \ s_1 \in S_1, s_s \in S_2 \}$ . Consider the commutative diagram with exact rows
 
 \[ 
\begin{tikzcd}
{ 1}\arrow{r}{}&  \frac{[F,F] \cap R}{[R,F]}  \arrow{r}
\arrow{d}{\alpha}
 &  \frac{ [F,F]}{[R,F]} \arrow{r}{} \arrow{d}{\beta}& \ \frac{[F,F]}{[F,F] \cap R} \arrow{r}{}\arrow{d}{\gamma}& { 1}\\%
{ 1}\arrow{r}{}& \ \frac{\overline{[\widehat{F},\widehat{F}]} \cap \overline{R}}{\overline{[\overline{R}, \widehat{F}]}} \arrow{r}{}  & \frac{ \overline{[\widehat{F},\widehat{F}]}}{\overline{[\overline{R},\widehat{F}]}}   \arrow{r}{} & \ \frac{\overline{[\widehat{F},\widehat{F}]}}{\overline{[\widehat{F},\widehat{F}]} \cap \overline{R}} \arrow{r}{}& { 1}\\%
\end{tikzcd}
\]
where all the maps are induced by the canonical map $F \to \widehat{F}$.

 Note that by the Hopf formula for abstract and profinite groups we have $$H_2(G, \mathbb{Z}) \simeq  \frac{[F,F] \cap R}{[R,F]} \hbox{ and }H_2(\widehat{G}, \widehat{ \mathbb{Z}}) \simeq \frac{\overline{[\widehat{F},\widehat{F}]} \cap \overline{R}}{\overline{[\overline{R}, \widehat{F}]}}$$ and $\alpha$ by assumption is injective.
 
 Note that $$ \frac{[F,F]}{[F,F] \cap R}  \simeq  \frac{[F,F] R}{R} \simeq [G,G]$$ is the commutator subgroup of $G.$  Similarly $$\frac{\overline{[\widehat{F},\widehat{F}]}}{\overline{[\widehat{F},\widehat{F}]} \cap \overline{R}} \simeq \frac{\overline{[\widehat{F},\widehat{F}]}\overline{R}}{\overline{R}} \simeq \overline{[\widehat{G}, \widehat{G}]}$$ is the commutator subgroup of $\widehat{G}.$
 Thus $\gamma$ can be identified with the map $\gamma_0 : [G,G] \to \overline{[\widehat{G}, \widehat{G}]}$ induced 
 by  the canonical map $\gamma_1 : G \to \widehat{G}$. Since $G$ is residually finite, $\gamma_1$ and hence $\gamma_0$ are  injective. Since $\alpha$ and $\gamma$ are injective we conclude that $\beta$ is injective.
 
 Finally by Theorem \ref{wedge} $$G \wedge G \simeq  \frac{ [F,F]}{[R,F]}, $$ hence via $\gamma$ we embed $G \wedge G$ in the profinite group $\frac{ \overline{[\widehat{F},\widehat{F}]}}{\overline{[\widehat{R},\widehat{F}]}} $ and every profinite group is residually finite.
 
  \end{proof}
  
  \begin{lemma}   Let $p$ be a prime integer.
  Suppose that $G$ is a discrete group such that $G$ is residually $p$-finite. Then
  
 a) If the map $H_2(G, \Z) \to \varprojlim_U H_2(G/U, \Z)$, induced by the canonical projections $G \to G/ U$ (for $ \widehat{G}_p \simeq \varprojlim G/U$, $G/ U$ $p$-finite groups) is injective then $G \wedge G$ is residually $p$-finite.
 
 b) If   the canonical map $H_2(G, \mathbb{Z}) \to H_2(\widehat{G}_p, {\mathbb{Z}}_p)$ is injective then $G \wedge G$ is residually $p$-finite.
 \end{lemma}
 
 \begin{proof} 
 The first proof of the previous result applies with $U$ running among all normal subgroups of $G$ such that $G/U$ is $p$-finite. Furthermore we have to substitute in the proof the profinite completion $\widehat{G}$ with the pro-$p$ completion $\widehat{G}_p$ and $\widehat{\mathbb{Z}}$ with $\mathbb{Z}_p$.
 \end{proof}
 
 \begin{lemma}    Let $p$ be a prime integer.
 
 a) If $G/ G'$ is finitely generated and either has no 2-torsion or  $G'$ has a complement in $G$ then $G \otimes G$ is residually finite if and only if $G \wedge G$ is residually finite.
 
 b)  If $G/ G'$ is finitely generated and either has no 2-torsion or  $G'$ has a complement in $G$ then $G \otimes G$ is residually $p$-finite if and only if $G \wedge G$ is residually $p$-finite and the torsion subgroup of $\Delta(G)$ is $p$-finite. In particular,   if $G/ G'$ is finitely generated and either  $G/ G'$ has no 2-torsion or  $G'$ has a complement in $G$, then   $G \otimes G$ is residually $p$-finite for every prime $p$ if and only if $\Delta(G)$ is torsion-free and $G \wedge G$ is residually $p$-finite  for every prime $p$.
 \end{lemma}
 
 \begin{proof}

  Let $C$ be a finitely generated abelian group and $\Gamma(C)$ be  the Whitehead quadratc functor. By \cite{M-M-O} there is an exact sequence
  $$ C \otimes C \to \Gamma (C) \to C/ 2 C \to 0$$ Then $\Gamma(C)$ is a finitely generated abelian group.
   
    By  \cite{B-L1}, \cite{B-L2} $\Delta(G)$ is a quotient of the Whitehead quadratic functor $\Gamma(G/ G')$. Thus for $C = G/ G'$ finitely generated  we obtain that $\Delta(G)$ is a finitely generated abelian group, hence is residually finite.
   
   By \cite{B-F-M}  if $G/ G'$ does not contain 2-torsion  or  $G'$ has a complement in $G$ then $$G \otimes G \simeq \Delta(G) \times (G \wedge G)$$

a)    The above isomorphism implies  that $G \otimes G$ is residually finite if and only if  $G \wedge G$ is residually finite. 

b) Finally note that a finitely generated abelian group is residually $p$-finite if and only if its torsion subgroup is a $p$-group. Then apply this for $\Delta(G)$.
   \end{proof}

  \begin{prop} \label{iff}
  Suppose that $G$ is a finitely generated residually finite group such that $G \wedge G$ is finitely generated. Then $G \wedge G$ is residually finite if and only if the $H_2(G, \mathbb{Z}) \to H_2(\widehat{G}, \widehat{\mathbb{Z}})$ induced by the canonical map  $G \to \widehat{G}$ is injective.
  \end{prop}
  
\begin{proof} Suppose that $G \wedge G$ is residually finite.  Consider the short exact sequence of discrete groups
 \begin{equation} \label{123}  1 \to \frac{[F,F]}{[R,F]} \to \frac{F}{[R,F]} \to \frac{F}{[F,F]} \to 1 \end{equation}
 Since $\frac{[F,F]}{[R,F]}  \simeq G \wedge G$ we have that  $\frac{[F,F]}{[R,F]}$ is finitely generated,  residually finite. Note that $\frac{F}{[F,F]}$ is a finitely generated abelian group, hence is  (cohomologically) good and residually finite. 
 
By Proposition \ref{SerSer} applied   to the short exact sequence (\ref{123}) we conclude that
  $\frac{F}{[R,F]}$ is residually finite and so the canonical map $ \frac{F}{[R,F]}   \to \frac{ \widehat{F}}{\overline{[\overline{R},\widehat{F}]}}$ is an inclusion. But the restriction of this map to 
 $ \frac{[F,F] \cap R}{[R,F]} $ is the map $\alpha$ from the second proof of Lemma \ref{res-finite}, so $\alpha$ is injective. Note that using the Hopf formula for discrete and profinite groups,  $\alpha$ can be identified with the canonical map $H_2(G, \mathbb{Z}) \to H_2( \widehat{G}, \widehat{\mathbb{Z}} )$.
\end{proof} 

\begin{cor} Suppose that $G$ is a finitely generated residually finite group such that the commutator subgroup $G'$ is finitely generated and $H_2(G, \mathbb{Z})$  is finitely generated. Then $G \wedge G$ is residually finite if and only if the map $H_2(G, \mathbb{Z}) \to H_2(\widehat{G}, \widehat{\mathbb{Z}})$ induced by the canonical map  $G \to \widehat{G}$ is injective.
\end{cor}
\begin{proof} There is a short exact sequence $$ 1 \to H_2(G, \mathbb{Z}) \to G \wedge G \to G' \to 1,$$ hence $G \wedge G$ is finitely generated. Then we can apply Proposition \ref{iff}.
\end{proof}
  
  \section{Homological goodness } \label{goodness}

 We define a new class of groups $\mathcal{A}_n^{hom}$ using the following homological definition:   a group $G $ belongs to the class $\mathcal{A}_n^{hom}$ if  the canonical map $G \to \widehat{G}$ induces an isomorphism $H_i(G, A) \to H_i(\widehat{G}, A) $ for every finite discrete  $\widehat{G}$-module $A$ and  $0 \leq i \leq n$.
 
  A discrete group $G$ is called {\it (cohomologically) good} if $G \in \mathcal{A}_n$  for every $n$ and it is is called {\it (homologically) good} if $G \in \mathcal{A}_n^{hom}$  for every $n$.

      \begin{lemma} \label{hom-goof-injective} If $G \in \mathcal{A}_2^{hom}$  and $H_2(G, \mathbb{Z})$ is finitely generated then  $H_2(G, \mathbb{Z}) \to H_2(\widehat{G}, \widehat{\mathbb{Z}})$ is injective. 
   \end{lemma}
   
   \begin{proof}
  By the universal coeficients theorem there is a short exact sequence
    $$0 \to H_2(G, \mathbb{Z}) \otimes_{\mathbb{Z}} \mathbb{Z}/ n \mathbb{Z} \to H_2(G,  \mathbb{Z}/ n \mathbb{Z}) \to Tor_1^{\mathbb{Z}}( H_1(G, \mathbb{Z}),  \mathbb{Z}/ n \mathbb{Z}) \to 0,$$
    in particular there is an exact sequence
    $$0 \to H_2(G, \mathbb{Z}) \otimes_{\mathbb{Z}} \mathbb{Z}/ n \mathbb{Z} \to H_2(G,  \mathbb{Z}/ n \mathbb{Z}) $$ Applying inverse limit over all positive integers $n$ we have an exact sequence
    \begin{equation} \label{eq111} 0 \to \varprojlim_n H_2(G, \mathbb{Z}) \otimes_{\mathbb{Z}} \mathbb{Z}/ n \mathbb{Z} \to \varprojlim_n H_2(G,  \mathbb{Z}/ n \mathbb{Z}) \end{equation}
    Since $G \in \mathcal{A}_2^{hom}$  there is an isomorphism $H_2(G,  \mathbb{Z}/ n \mathbb{Z}) \simeq H_2(\widehat{G},  \mathbb{Z}/ n \mathbb{Z})$, hence
    \begin{equation} \label{eq222}  \varprojlim_n H_2(G,  \mathbb{Z}/ n \mathbb{Z}) \simeq \varprojlim_n H_2(\widehat{G},  \mathbb{Z}/ n \mathbb{Z}) \simeq H_2(\widehat{G}, \varprojlim_n \mathbb{Z}/ n \mathbb{Z}) \simeq H_2(\widehat{G}, \widehat{\mathbb{Z}}). 
    \end{equation}
    We have used above $ \varprojlim_n \mathbb{Z}/ n \mathbb{Z} \simeq \widehat{\mathbb{Z}}$.
    Note that 
   $$\varprojlim_n H_2(G, \mathbb{Z}) \otimes_{\mathbb{Z}} \mathbb{Z}/ n \mathbb{Z} \simeq \varprojlim_n H_2(G, \mathbb{Z}) / n H_2(G, \mathbb{Z})$$  hence the map
   $$ H_2(G, \mathbb{Z}) \to \varprojlim_n H_2(G, \mathbb{Z}) \otimes_{\mathbb{Z}} \mathbb{Z}/ n \mathbb{Z}$$ has kernel $ \cap_n n H_2(G, \mathbb{Z})$.
   Since $H_2(G, \mathbb{Z}) $ is a finitely generated abelian group we have that 
   $\cap_n n H_2(G, \mathbb{Z}) = 0$. 
   Then (\ref{eq111}), (\ref{eq222})  complete the proof.
   \end{proof}

 \begin{cor}[= Corollary B] \label{equi}
If $G \in \mathcal{A}_2^{hom}$, $G$ is residually finite and $H_2(G, \z)$ is finitely generated then the map $H_2(G, \Z) \to H_2( \widehat{G}, \widehat{\Z})$ is injective and $G \wedge G$ is residually finite.
\end{cor}

 \section{Metabelian groups  and central-by-metabelian groups} \label{meta}

 \begin{theorem} \label{fin-pres1}
 Let $G$ be a finitely presented metabelian group. Then
 
 a) the canonical map $H_2(G, \mathbb{Z}) \to H_2(\widehat{G}, \widehat{\mathbb{Z}})$ is injective;

 b)  $G \wedge G$ is residually finite. 
 \end{theorem}
 
 \begin{proof}
 a) Let $G = F/ R$ be a metabelian group with $F$ finitely generated free group. Then
 $$H_2(G, \mathbb{Z}) \simeq R \cap F'/ [F,R]$$
 Consider $B = F/ [F,R]$.  Since $G$ is finitely presented,
 $R$ is finitely generated as a normal subgroup of $F$. Hence $R/ [R, F]$ is a finitely generated abelian group, so it is finitely presented. Then
  $B$ is finitely presented, central-by-metabelian and by the main result of \cite{Groves}  $B$ is abelian-by-polycyclic. By the 
 Roseblade-Jategaonkar result \cite{J}, \cite{R} finitely generated abelian-by-polycyclic groups are residually finite, in particular $B$ is residually finite.

 Let $\widehat{F}$ be the profinite completion of $F$ and for a subgroup $M$ of $F$ write  $\overline{M}$ be the closure of $M$ in $\widehat{F}$. Then we have short exact sequences of groups
 $$1 \to R/[R, F] \to B \to G \to 1
 \hbox{ and } 1 \to \overline{R}/ \overline{[R,F]} \to \widehat{B} \to \widehat{G} \to 1$$
 To justify the last exact sequence it suffices to consider  short exact sequences
 $$ 1 \to \overline{R} \to \widehat{F} \to \widehat{G} \to 1 \hbox{ and }  1 \to \overline{[R,F]} \to \widehat{F} \to \widehat{B} \to 1$$
 The first exact sequence induces a short  exact sequence
$  1 \to \overline{R}/  \overline{[R,F]} \to \widehat{F}/ \overline{[R,F]} \to \widehat{G} \to 1$, that together with the isomorphism $\widehat{F}/ \overline{[R,F]} \simeq \widehat{B}$, completes the argument of the exactness.

 Note that $\overline{[\overline{R}, \widehat{F}]} = \overline{[R,F]}$.
 Suppose that the map
 $$R \cap F'/ [R,F] \simeq H_2(G, \mathbb{Z}) \to H_2(\widehat{G}, \widehat{\mathbb{Z}}) \simeq \overline{R} \cap  \overline{ [\widehat{F}, \widehat{F}]} / \overline{[\overline{R}, \widehat{F}]}$$
 is  not injective. Then the map
 $R/ [R, F] \to \overline{R}/ \overline{[R,F]}$, is not injective, hence $B \to \widehat{B}$ is not injective i.e. $B$ is not residually finite,
a contradiction. Hence $H_2(G, \mathbb{Z}) \to H_2(\widehat{G}, \widehat{\mathbb{Z}})$ is injective.
 
b) Finally note that by Hall result finitely generated metabelian groups are residually finite. Then by Theorem A, a)  $G \wedge G$ is residually finite.
 
 There is an alternative proof of b) that does not use Theorem A but uses the description of $G \wedge G$ given by Theorem \ref{wedge} i.e. $G \wedge G \simeq [F,F]/ [R, F] = [B, B]$ and the fact proved above that $B$ is residually finite, so its subgroup $[B,B]$  is residually finite.
 \end{proof}
 
Let $\sigma: \mathbb{Z} Q \to \mathbb{Z} Q$ be the $\mathbb{Z}$-linear map that sends $q \in Q$ to $q^{-1}$.
 
  \begin{theorem} \label{a1}
 Let $1 \to A \to G \to Q \to 1$ be  a short exact sequence of groups, $A$ and $Q$  abelian and $G$ finitely generated. Suppose that for $I = ann_{\mathbb{Z} Q} A$ we have that $\mathbb{Z} Q/ (I + \sigma(I))$ is finitely generated as an additive group. Then

 a) the canonical map $H_2(G, \mathbb{Z}) \to H_2(\widehat{G}, \widehat{\mathbb{Z}})$ is injective;

 b)  $G \wedge G$ is residually finite. 
 \end{theorem}
 
 \begin{proof} The proof is the same as the proof of Theorem \ref{fin-pres1}, we need only that $B = F/ [F,R]$ is finitely generated abelian-by-polycyclic group. This is proved in Lemma \ref{sexta}.
 \end{proof}
 
 \begin{lemma} \label{sexta} Let $1 \to M \to H \to G \to 1$ be a short exact sequence of groups with $M$ central in $H$, $H$ finitely generated. Let  $1 \to A \to G \to Q \to 1$ be  a short exact sequence of groups, $A$ and $Q$  abelian.  Suppose that for $I = ann_{\mathbb{Z} Q} A$ we have that $\mathbb{Z} Q/ (I + \sigma(I))$ is finitely generated as an additive group. Then $H$ is abelian-by-polycyclic.
 \end{lemma}
 
 \begin{proof} Let $A_0$ be the preimage of $A$ in $H$ and $C = Z(A_0)$. Note that $M \subseteq C$ and   let $V = A_0/ C$. 
 We view $V$ as a $\mathbb{Z} Q$-module, where the $Q$-action is induced by conjugation.
 
 Note that for $h \in H$, $a_1, a_2 \in A_0$  we have
 $$[a_1^h, a_2] = [a_1, a_2^{h^{-1}}]^h = [a_1, a_2^{h^{-1}}]$$ since $[a_1, a_2^{h^{-1}}] \in A_0' \subseteq M \subseteq Z(H)$.
 
 Let $J = ann_{\mathbb{Z} Q} (V)$. The equality  $[a_1^h, a_2] = [a_1, a_2^{h^{-1}}]$
 implies that if $\lambda \in J$ then $\sigma(\lambda) \in J$. 
 
 Note that $V$ as a $\mathbb{Z}Q$-module is a quotient of $A$, hence $I \subseteq J$ and so $\sigma(I) \subseteq \sigma(J) = J$ and $I + \sigma(I) \subseteq J$. Thus $\mathbb{Z} Q/ J$  is a quotient of  $\mathbb{Z} Q/ (I + \sigma(I))$, so $\mathbb{Z} Q/ J$  is finitely generated as an additive group.
 
 Note that $A$ is finitely generated as $\mathbb{Z} Q$-module, hence its quotient $V$ is finitely generated as $\mathbb{Z} Q$-module. This combined with the fact that  $\mathbb{Z} Q/ J$ is finitely generated as an additive group implies that  $V$ is finitely generated as an additive group. Thus $H/ C$ is polycyclic and $C$ is abelian.
 \end{proof}
 
 \medskip
 {\bf Examples} 1) Consider the group $G = \mathbb{Z}[\frac{1}{6}] \rtimes \mathbb{Z}$, where $\mathbb{Z} = Q$ has a generator $x$ that acts on $A = \mathbb{Z}[\frac{1}{6}] $ by multiplication by $\frac{2}{3}$. By the Bieri-Strebel criterion Theorem \ref{classification} $G$ is not finitely presented. Here $I = ( 3x - 2) = ann_{\mathbb{Z} Q} A$. Note that $I + \sigma(I) = (3x-2,3x^{-1} - 2) = ( 3 x - 2, 3 - 2x) $ and $- 13 = 2(3x - 2) - 3( 3 - 2x) \in I + \sigma(I)$ implies that $\mathbb{Z} Q/(I + \sigma(I)) \simeq \mathbb{Z}_{13}$.
 
 By Theorem \ref{a1} $G \wedge G$ is residually finite.
 
 Note that $G'= [A, Q] = A ( \frac{2}{3} - 1) = A$, so $G/ G'\simeq Q \simeq \mathbb{Z}$ has no 2-torsion. Then by Theorem A, d) $G \otimes G$ is residually finite. By Proposition \ref{ab-pol1}   $H_2(G, \mathbb{Z})$ is finitely generated.
 
\medskip 
 2) Let $Q = \mathbb{Z} = \langle x \rangle$, $A = \mathbb{Z} Q$, $G = A \rtimes Q$.  The conjugation action of $x$ on $A$ corresponds to multiplication with $x$ in $A$. Note that $G$ is not finitely presented, so we cannot apply Theorem \ref{fin-pres1} and since $ann_{\mathbb{Z} Q} (A) = 0$ we cannot apply Theorem \ref{a1}.
 
 Consider the LHS spectral sequence $E_{p,q}^2 = H_p(Q, H_q(A, \mathbb{Z}))$ converging to $H_2(G, \mathbb{Z})$. Then  $E_{2,0}^2 = 0$  and  $$E_{1,1}^2 = H_1(Q, H_1(A, \mathbb{Z})) = H_1(Q, A) \simeq H^0 (Q, A) = 0$$ the spectral sequence collapses and so $$H_2(G, \mathbb{Z}) \simeq E_{0,2}^2 = H_0(Q, H_2(A, \mathbb{Z})) \simeq H_0(Q, A \wedge A)$$
 
 Note that $\widehat{G} = B \rtimes \widehat{Q}$, where $B = \widehat{\mathbb{Z}}[[\widehat{Q}]]$  is the completed group algebra. Similarly to the above
 $$H_2(\widehat{G}, \widehat{\mathbb{Z}}) \simeq  H_0(\widehat{Q}, H_2(B, \widehat{\mathbb{Z}})) \simeq H_0(\widehat{Q}, B \widehat{\wedge} B)$$
Note that we have used that $H_2(B, \widehat{\mathbb{Z}}) \simeq B \widehat{\wedge}_{\widehat{\Z}} B$, that follows from  Lemma \ref{wedge2}.
 
 Note that $A \wedge A = \oplus_{q_1 < q_2, q_i \in Q} \mathbb{Z} q_1 \wedge q_2$ where $q_1 < q_2$ if $q = q_2 q_1^{-1} = x^i$ for some $i > 0$. Then $$H_0(Q, A \wedge A) \simeq \oplus_{q \in Q, q > 1}  \mathbb{Z} (1 \wedge q)$$
 
 On other hand $$ B \widehat{\wedge} B \simeq  \varprojlim_n \mathbb{Z}_n [[Q / Q^n]] \wedge  \mathbb{Z}_n [[Q / Q^n]]$$
 where $\mathbb{Z}_n = \mathbb{Z}/ n \mathbb{Z}$ and
 $$ H_0(\widehat{Q}, B \widehat{\wedge} B) \simeq \varprojlim_n H_0(Q/ Q^n, \mathbb{Z}_n [[Q / Q^n]] \wedge  \mathbb{Z}_n [[Q / Q^n]]) \simeq \varprojlim_n \oplus_{ 0 < i < n/2}  \mathbb{Z}_n (1 \wedge x^i)$$

Suppose that $u : =  \sum_{i > 0} z_i ( 1 \wedge x^i) \in Ker (H_2(G, \mathbb{Z}) \to H_2( \widehat{G}, \widehat{\mathbb{Z}})) = Ker (H_0(Q, A \wedge A) \to H_0(\widehat{Q}, B \widehat{\wedge} B) $, where all $z_i \in \mathbb{Z}$. 
Then $$\sum_{i > 0} z_i ( 1 \wedge x^i) \in  Ker (H_0(Q, A \wedge A) \to H_0(Q/ Q^n, \mathbb{Z}_n [[Q / Q^n]] \wedge  \mathbb{Z}_n [[Q / Q^n]]) $$

 Let $n>0$ be such that $z_i = 0$ for $i \geq n/2$. Hence by the above description we get that $n$ divides each $z_i$. Since this holds for infinitely many $n$ we conclude that each $z_i = 0$. Hence the map $H_2(G, \mathbb{Z}) \to H_2( \widehat{G}, \widehat{\mathbb{Z}})$ is injective and so by Theorem A, $G \wedge G$ is residually finite.

 \medskip 
 \begin{lemma} \label{wedge2} Suppose $B$ is an abelian torsion-free profinite group. Then 
 
 a) $H_2(B, \widehat{\Z}) \simeq B \widehat{\wedge}_{\widehat{\Z}}  B$,
 
 b)  $H_2(B, {\Z}/ q \Z) \simeq (B/ q B)  \widehat{\wedge}_{\widehat{\Z}} (B/ q B)$ for $q \geq 2$.
 \end{lemma}
 
 \begin{proof} a) Let $B \simeq \varprojlim_i M_i$
 where $M_i$ are finite quotients of $B$.
 By the universal coefficients theorem we have an exact sequence
 $$0 \to H_2(M_i, \Z) \otimes (\Z/ m \Z) \to H_2(M_i, \Z/ m \Z) \to Tor^{\Z}_1( \Z/ m \Z, M_i) \to 0$$ Applying inverse limit we get an exact sequence
 $$0 \to  \varprojlim_{i,m} H_2(M_i, \Z) \otimes (\Z/ m \Z) \to \varprojlim_{i,m}   H_2(M_i, \Z/ m \Z) \to \varprojlim_{i,m}   Tor^{\Z}_1( \Z/ m \Z, M_i) $$ $$ \to  \varprojlim_{i,m} \ ^1 H_2(M_i, \Z) \otimes (\Z/ m \Z) = 0$$
 where the factor $\varprojlim^1 = 0$ since the corresponding inverse system is of finite modules so the Mittag-Leffler condition holds.
 Since $H_2(M_i, \Z) \simeq M_i \wedge_{\Z} M_i$ we get
 $$ \varprojlim_{i,m} H_2(M_i, \Z) \otimes (\Z/ m \Z) \simeq  \varprojlim_{i,m}  M_i \wedge_{\Z} M_i / m  (M_i \wedge_{\Z} M_i) \simeq $$ $$\varprojlim_{i,m} (M_i/mM_i) \wedge_{\Z} (M_i/ m M_i)  \simeq B \widehat{\wedge}_{\widehat{\Z}}  B$$
 The last isomorphism follows from $B \simeq \varprojlim_{i,m} M_i/ m M_i$.
 
 It remains to show that \begin{equation} \label{eqeqeq}  \varprojlim_{i,m}   Tor^{\Z}_1( \Z/ m \Z, M_i) = 0 \end{equation}
  Suppose that   $a = (  a_{i,m} )_{i,m} \in \varprojlim_{i,m}   Tor^{\Z}_1( \Z/ m \Z, M_i)$ and that $a$ is non-zero, so there is some $a_{i,m} \not= 0$. Note that $Tor^{\Z}_1( \Z/ m \Z, M_i) \simeq \{ v \in M_i \ | \ m v = 0 \}$. Then $Tor^{\Z}_1( \Z/ m \Z, M_i)$ is an abelian subgroup of $M_i$, hence 
  $b = (a_{i,m} )_i \in \varprojlim_{i}   Tor^{\Z}_1( \Z/ m \Z, M_i)$ is an abelian subgroup of $\varprojlim_{i}  M_i = B$. Then $m a_{i,m} = 0$, so $m b = 0$. But $B$ 
  is torsion-free, so $b = 0$ and $a_{i,m} = 0$ a contradiction.

 b) If in the previous argument we fix $m = q$ we get that $$ \varprojlim_{i}   Tor^{\Z}_1( \Z/ q \Z, M_i) = 0 $$
  This together with the universal coefficients theorem gives
 $$H_2(B, \Z/ q \Z) \simeq \varprojlim_{i}   H_2(M_i, \Z/ q \Z) \simeq   \varprojlim_{i} H_2(M_i, \Z) \otimes (\Z/ q \Z) \simeq $$ $$ \varprojlim_{i} (M_i \wedge_{\Z} M_i) \otimes (\Z/ q \Z) \simeq  \varprojlim_{i} (M_i/ q M_i \wedge_{\Z} M_i/ q M_i) \simeq (B/qB) \widehat{\wedge}_{\widehat{\Z}} (B/ q B)$$ 
 \end{proof}

  \begin{prop} \label{ab-pol1} Let $G$ be a finitely generated group, $F$ a finitely generated free group such that $G \simeq F/ R$. If
$F/[F,R]$ is abelian-by-polycyclic, then $H_2(G, \mathbb{Z})$ is finitely generated. 
\end{prop}

\begin{proof}
Let $M_0$ be a normal subgroup of $F$ such that 
$ M_0/[F,R]$ is abelian and $F/ M_0$ is polycyclic, hence $F/ M_0$ is finitely presented. Then 
$ M_0/[F,R]$ is a finitely generated $\mathbb{Z} (F/ M_0)$-module via conjugation. Since $\mathbb{Z} (F/ M_0)$ is a Noetherian ring we conclude that $ M_0/[F,R]$ is a noetherian $\mathbb{Z} (F/ M_0)$-module i.e. every submodule is finitely generated. In particular
$ R \cap M_0/[F,R]$ is finitely generated as $\mathbb{Z} (F/ M_0)$-module.

Consider the short exact sequence of groups
$$1 \to R \cap M_0/[F,R] \to R/[F,R] \to R/R \cap M_0
\to 1$$
We note that $ R/R \cap M_0 \simeq RM_0/M_0$ is a subgroup of the polycyclic group $ F/ M_0 $, hence $ R/R \cap M_0$ is polycyclic, so it is finitely generated.This combined with the above short exact sequence implies that $ R/[F,R]$ is finitely generated as $\mathbb{Z} (F/ M_0)$-module. But the action ( by conjugation) of $F$ on $ R/[F,R]$ is the trivial one, hence  $ R/[F,R]$ is finitely generated as abelian group. Hence its subgroup $R \cap[F,F]/[F,R]$ is finitely generated as an abelian group.

Finally we use the Hopf formula $H_2(G,\mathbb{Z}) =
R \cap[F,F]/[F,R]$ to deduce  $H_2(G, \mathbb{Z})$ is finitely generated.
\end{proof}

We present first an easier proof of Theorem D when $G$ is metabelian and in Section \ref{sec-final} we prove the case of centre-by-abelian group using a more complicated proof.

  \begin{theorem} \label{nu}
 Let $G$ be a finitely presented metabelian group. Then $\nu(G)$  is a finitely generated abelian-by-polycyclic group, hence it is residually finite. In particular $G \otimes G$ is residually finite.
 \end{theorem}
 
 \begin{proof} Note that $$\nu(G)/ W_0(G) \simeq  \rho_0(G) \subseteq G \times G \times G$$ is metabelian. The group  $W_0(G)$ was defined in preliminary section \ref{intro-non-abelian}, it  is a normal subgroup of $\nu(G)$  and was denoted by $\mu(G)$ in \cite{Norai3}. By \cite{Norai3} $$W_0(G)/ \Delta(G) \simeq W(G)/ R(G)\hbox{ and }W_0(G) \hbox{ is central in }\nu(G)$$
 
  By \cite{BHMS} if a subdirect product of finitely presented groups virtually surjects on pairs then it is finitely presented itself. This applies for $\rho_0(G)$ considered as a subgroup of $G \times G \times G$ and we conclude that $\rho_0(G)$ is a finitey presented metabelian group i.e. $W_0(G)$ is finitely generated as a normal subgroup of $\nu(G)$. Since $W_0(G)$ is central we conclude that $W_0(G)$ is a finitely generated abelian group, hence is finitely presented. Combining with $\nu(G) / W_0(G)$ is finitely presented we conclude that $\nu(G)$ is finitely presented. By the Groves result a finitely presented centre-by-metabelian group is abelian-by-polycyclic \cite{Groves}, in particular $\nu(G)$ is abelian-by-polycyclic.
  Hence $\nu(G)$ is residually finite and its subgroup $[G, G^{\varphi}] \simeq G \otimes G$ is residually finite.
 
 \end{proof}
 
 \begin{theorem}[=Theorem D] Let $G$ be a finitely presented metabelian group. Then $\X(G)$ is residually finite.
 \end{theorem}
 
 \begin{proof}  Recall that by section \ref{intro-non-abelian} $$W(G) = L(G) \cap D(G)$$ is a normal subgroup of $\X(G)$, $$[L(G), D(G)] = 1$$ by \cite[Prop. 4.1.7]{Said}, hence $W(G)$ commutes with $L(G) D(G)$. Furthermore $$\X(G) / L(G) D(G) \simeq G/ G' \hbox{ is abelian,}$$ hence $\X(G)' \subseteq L(G) D(G)$ and the commutator subgroup $\X(G)'$ acts trivially ( via conjugation) on $W(G)$.
 Note that $W(G)$ is an abelian normal subgroup of $\X(G)$ and $W(G) = Ker (\rho)$ where $$\rho : \X(G) \to G \times G \times G$$ is defined in section \ref{intro-non-abelian}. 
 Then
 $$\X(G)/ W(G) \simeq Im (\rho) \leq G \times G \times G,$$  hence $\X(G)/ W(G)$  is metabelian and finitely presented since $Im (\rho)$ is a subdirect product in $G \times G \times G$ that maps surjectively on pairs, hence by \cite[Thm. A]{BHMS} such subdirect products are finitely presented. The fact that $\X(G)$ is finitely presented is proved in \cite{B-K}.
 
 Consider the class $\mathcal{H}$ of all finitely presented groups $H$ with an abelian normal subgroup $W$ such that $H/ W$ is finitely presented metabelian and $H'$-acts trivially on $W$ (via conjugation). Thus $W$ is finitely generated as a $\mathbb{Z} H/ H'$-module with $H/ H'$ action induced by conjugation. Since $\mathbb{Z} H/ H'$ is a Noetherian ring and any quotient of a finitely presented metabelian group is finitely presented \cite{B-S}, we deduce that $H$ has max-n i.e. every normal subgroup of $H$ is finitely generated as a normal subgroup.   In particular
  the class of groups $\mathcal{H}$ is quotient closed. 
  
  A group is said to be monolithic if the intersection of all non-trivial normal subgroups is non-trivial.  By Lemma 1 from \cite{Hall2} a group in a class of groups, where the class of groups  is quotient closed,  is residually finite if and only if every monolithic group in this class is finite. 
 
 We  aim to show that every monolithic  group $H$ in the class $\mathcal{H}$ is finite. Let $W$ be the subgroup of such monolithic group $H$ given by the definition of the class $\mathcal{H}$.
 Note that $W$ is a monolithic $\mathbb{Z} H$-module i.e. the intersection of all non-trivial $\mathbb{Z} H$-submodules is non-trivial. Note that the $H$-action on $W$ via conjugation factors through action of the abelian group $Q = H / H'$. By Theorem 1* from the Hall paper \cite{Hall2} any monolithic finitely generated abelian-by-nilpotent group is finite, in particular  the finitely generated metabelian monolithic group $W \rtimes (Q / C_Q(W))$  is finite. Then $W$ is finite and there is a normal subgroup  $S = Ker(H \to Aut(W))$ of finite index in $H$ such that $W \subseteq Z(S)$ i.e. $W$ is central in $S$.   Note that  since $S$ has finite index in $H$, $S$ is a finitely presented centre-by-metabelian group, hence by the Groves result \cite{Groves} $S$ is abelian-by-polycyclic. Thus $H$ is abelian-by-polycyclic-by-finite. But abelian-by-polycyclic-by-finite groups are residually finite, hence a monolithic abelian-by-polycyclic-by-finite group is finite, in particular $H$ is finite.
 
 \end{proof}
 
 The idea of the above proof can be used to show directly that for a finitely presented metabelian group $G$ we have that $\nu(G)$ is residually finite. But in Theorem \ref{nu} we proved  more, by showing that $\nu(G)$ is abelian-by-polycyclic.
 
In several proofs we used that under some conditions centre-by-metabelian groups are abelian-by-polycyclic and hence residually finite. In \cite{Groves2} were constructed finitely generated centre-by-metabelian groups that are not residually finite. This supports the need of extra conditions in the previous statements. The examples in \cite{Groves2} have infinite cyclic central subgroup with metabelian quotient that is not finitely presented.

 Recall that a $p$-group is one where every non-trivial element has order a power of $p$.

 \begin{cor}
 Let $G$ be a finitely presented metabelian group.
 
 a) Suppose that  $G'$, $G/ G'$ and $H_2(G, \mathbb{Z})$ are $\mathbb{Z}$-torsion-free. Then for almost all primes $p$ $\nu(G)$ and $G \otimes G$ have a normal subgroup of finite index that is residually a  finite $p$-group.
 
 b)  Suppose that  $G'$ and $H_2(G, \mathbb{Z})$ are $\mathbb{Z}$-torsion-free. Then   for almost all primes $p$ $G \wedge G$ has a normal subgroup of finite index that is residually a finite $p$-group.
 
 c) Suppose that  $G'$ and $H_2(G, \mathbb{Z})$ are $p$-groups. Then   for almost all primes $p$ $G \wedge G$  is virtually residually $p$-finite.
 \end{cor}
 
 \begin{proof} We show that in all cases we can apply Theorem \ref{segal}.
 
a)  By Theorem \ref{nu} $\nu(G)$ is abelian-by-polycyclic i.e. there is an abelian normal subgroup $A$ of $\nu(G)$ such that $\nu(G)/ A$ is polycyclic. By the proof of Theorem \ref{nu} 
  $A$ is a subgroup of the group $$\rho_0^{-1} (G'\times G'\times G') = \nu(G)'$$ Note that
 $ W_0 = W_0(G)  \subseteq \nu(G)'$ and $$\nu(G)'/ W_0(G) \simeq G'\times G'\times G' \hbox{ is  }\mathbb{Z} \hbox{-torsion-free}$$ Furthermore $\Delta(G) \subseteq W_0$ and $W_0/ \Delta(G) \simeq H_2(G, \mathbb{Z})$ is  $\mathbb{Z}$-torsion-free. Finally since $G/ G'$ has no 2-torsion, $$\Delta(G) 
 \simeq \Delta(G/ G') = Ker (G/ G'\otimes G/ G' \to G/ G'\wedge G/ G')$$ is  $\mathbb{Z}$-torsion-free since $G/ G'\otimes G/ G'$ is $\mathbb{Z}$-torsion-free. Hence $ \nu(G)'$ has no torsion elements, hence $A$ is $\mathbb{Z}$-torsion-free. Then we can apply Theorem \ref{segal} ii).

b)   Note that  $G \wedge G$ is a subgroup of $H = \nu(G)/ \Delta(G)$. Recall that by Theorem \ref{nu} $\nu(G)$ is abelian-by-polycyclic, hence $\nu(G)/ \Delta(G)$ is abelian-by-polycyclic i.e. there is  an abelian normal subgroup $A_0$ of $H$ such that $H/ A_0$ is polycyclic. Furthermore $A_0$ is a subgroup of $\nu(G)'/ \Delta(G)$ and 
 the exact sequence  $1 \to W_0/ \Delta(G) \to \nu(G)'/ \Delta(G) \to \nu(G)'/ W_0 \to 1$ can be rewritten as a short exact sequence of groups  \begin{equation} \label{ses} 1 \to H_2(G, \mathbb{Z}) \to \nu(G)'/ \Delta(G) \to G'\times G'\times G'\to 1\end{equation} 
Thus $A_0$ is  $\mathbb{Z}$-torsion-free.  Then we can apply Theorem \ref{segal} ii).

c) We can define $A_0$ as in b).  The short exact sequence (\ref{ses}) implies that $A_0$ is a $p$-group.  Then we can apply Theorem \ref{segal} i).

 \end{proof}
 
  \section{ Proof of Theorem C} \label{sec-final}

  \begin{theorem}  \label{ab-pol} Let $G$ be a finitely presented group with a normal subgroup $A$ such that

1) $G/A$ is abelian;

2) $A'$ is  finitely generated abelian;

3) $[[A', G], G] = 1.$

Then $G$ is abelian-by-polycyclic and in particular, it is residually finite.

Furthermore if $A$ is characteristic in $G$ then there is a characteristic  abelian  subgroup $C$ of $G$ such that $G/ C$ is polycyclic.
  \end{theorem}

 \begin{proof}
  Set $M = [A', G]$ and let $B$ be the subgroup of $A$ that contains $M$ and $B/ M$ is the centre of $A/ M$. Thus $[A, B] \subseteq M$ and $A' \leq B$. Note that $M$ is normal in $G$ and since $B/M$ is characteristic in $A/M$, $B/ M$ is normal in $G/ M$, hence $B$ is normal in $G$. By assumption    $[M, G] = 1$ and since $M$ is a subgroup of the finitely generated abelian group $A'$ we have that $M$ is finitely generated abelian. Then $G/ M$ is a finitely presented, centre-by-metabelian group.   By \cite{B-S} if a finitely presented group does not have non-cyclic free subgroups then every metabelian quotient is finitely presented. In our case this implies that the central part of $G/ M$ is a finitely generated abelian group. This combined with every finitely presented metabelian group has max-$n$ implies $G/ M$ has max-n and every quotient of $G/ M$ is finitely presented. By the proof of \cite[Proposition]{Groves} $A/B$ is finitely generated.

  Let $C$ be the centre $Z(B)$ of $B$,  hence $M \leq C$.  
  We will show that $G/ C$ is polycyclic and by construction $C$ is normal abelian subgroup of $G$, so $G$ is abelian-by-polycyclic.  We observe that if $A$ is a characteristic subgroup of $G$ then $M$ is characteristic in $G$, hence $B$ is characteristic in $G$ and finally $C = Z(B)$ is characteristic in $G$.

  Set $V = B/ C$  and consider it as $G$-module via conjugation ( on the right). Since $[B,A] \leq M \leq C$, the $G$-action (via conjugation)  on $V$ factors through a $Q$-action, where $Q = G/ A$ is finitely generated abelian.  Since $G/ M$ has max-n, $V$ is a finitely generated $Q$-module.

  For $a \in A, b \in B, g \in G$, using that $[[A,B],G] \subseteq [M,G] = 1$ we have
  $$[a^g,b] = [a, b^{g^{-1}}]^g =  [a, b^{g^{-1}}]$$ In particular applying this for $a,b \in B$ and working modulo $C$  we have that if $\lambda \in I = ann_{\mathbb{Z} Q}(V)$ then $\sigma(\lambda) \in I$, where $$\sigma : \mathbb{Z} Q \to \mathbb{Z} Q$$  is the $\mathbb{Z}$-linear map  that sends $q \in Q$ to $q^{-1}$. 
  
\iffalse{  Since $A/B$ is finitey generated, we have that there are finitely many elements $x_1, \ldots, x_m $ in $A$ such that $A/ C = V \langle x_1, \ldots, x_m \rangle$. Then since $V$ is central in $A/ C$ we have $$(A/C)'= \langle [x_i, x_j] \ | \ 1 \leq i < j \leq m \rangle\hbox{ is  finitely generated}$$ Hence the group $K = G/ A'C$ is  finitely presented. By construction $K$ is metabelian. }\fi
  
  Recall that $A'$ is finitely generated. Then since $G$ is finitely presented we have that $G/ A'$ is finitely presented metabelian. Every quotient of a finitely presented metabelian group is finitely presented. Thus $$K = G/ A'C \hbox{ is  a finitely presented, metabelian group.}$$

  Set $V_0 = B/ A' C$ i.e. the image of $V$ in $K$. Our aim is to show that $V$ is finitely generated. But since  $(A/C)'$ is finitely generated, $V$ is finitely generated  precisely when  $V_0$ is finitely generated. Set $D = A/ A' C$ and note that $D$ is abelian and $V_0 \subseteq D \subseteq K$. Since $A/ B$ is finitely generated we have that $D/ V_0$ is finitely generated. We view $D$ as a finitely generated right $Q \simeq K/ D$-module via conjugation and aim to show that $D$ is finitely generated as an abelian group.
  
  Since $K$ is finitely presented by the classification of the finitely presented metabelian groups, see Theorem \ref{classification}, we have that $$\Sigma_D^c(Q) = S(Q) \setminus \Sigma_D(C)$$ has no antipodal points. On other hand the short exact sequence of $\mathbb{Z} Q$-modules
  $$0 \to V_0 \to D \to D/ V_0 \to 0$$ implies
$$\Sigma_D^c(Q) = \Sigma_{V_0}^c(Q) \cup \Sigma_{D/ V_0}^c(Q).$$ Since  $D/ V_0$ is finitely generated as an abelian group $\Sigma_{D/ V_0}^c(Q)  = \emptyset$, hence $\Sigma_D^c(Q) = \Sigma_{V_0}^c(Q)$  has no antipodal points.
By the short exact sequence of $\mathbb{Z} Q$-modules $$ 0 \to V_1 \to V \to V_0 \to 0$$ where $V_1 = A'C/ C$, we have 
$$\Sigma_{V_0}^c(Q) \cup \Sigma_{V_1}^c(Q) = \Sigma_V^c(Q)$$ Since $A'$ is finitely generated, $V_1$ is finitely generated as an abelian group, so $ \Sigma_{V_1}^c(Q) = \emptyset$. Thus $$\Sigma_{V_0}^c(Q) = \Sigma_V^c(Q)$$  has no antipodal points.

By Theorem \ref{ttt} $$\Sigma_V^c(Q) = \Sigma_{\mathbb{Z} Q/ I}^c(Q).$$ Recall that $\sigma (I) = I$, where $I = ann_{\mathbb{Z} Q}(V)$.  This implies that  $$ \Sigma_{\mathbb{Z} Q/ I}^c(Q) = - \Sigma_{\mathbb{Z} Q/ I}^c(Q)$$ Since  $ \Sigma_{\mathbb{Z} Q/ I}^c(Q) $ has no antipodal points we conclude that $$\Sigma_V^c(Q)^c =  \Sigma_{\mathbb{Z} Q/ I}^c(Q)  = \emptyset.$$ This is equivalent with $V$ is finitely generated as an abelian group ( see \cite[Thm A,i)]{B-S}). 

Finally since $G/A, A/ B$ and $B/ C$ are finitely generated abelian groups, the group $G/ C$ is polycyclic.
 \end{proof} 

As a corollary of Theorem \ref{ab-pol} we obtain the following result.

\begin{cor} [=Theorem C] \label{ThmDThm} 
Let $G_0$ be a finitely presented centre-by-metabelian group. Then $\nu(G_0)$ is  a finitely generated, abelian-by-polycyclic group  with the abelian normal subgroup being characteristic. In particular $\nu(G_0)$, $G_0 \otimes G_0$ and $G_0 \wedge G_0$ are residually finite. 
\end{cor} 

\begin{proof}
By assumption there is a normal subgroup $K$ of $G_0$ such that $G_0/ K$ is metabelian, $K \subseteq G_0'$ and $K \subseteq Z(G_0)$. By \cite{B-S} every metabelian quotient of a finitely presented soluble group is finitely presented. In particular $K$ is finitely generated as a normal subgroup of $G_0$ but as $K$ is central, we get that $K$ is finitely generated.

\medskip
{\bf Claim} {\it Let $B$ be the normal closure  of $K \cup K^{\varphi}$ in $\nu(G_0)$. Then

a) $[K, G_0^{\varphi}], [K^{\varphi}, G_0] \subseteq Z(\nu(G_0))$,

b) $B = \langle K, K^{\varphi}, [K, G_0^{\varphi}], [K^{\varphi}, G_0] \rangle$,

c) $B$ is abelian,

d) $[B,\nu(G_0), \nu(G_0)] = 1$.
} 

\medskip
Proof of the Claim. 

a) We will prove that $[K, G_0^{\varphi}] \subseteq Z(\nu(G_0))$, then  $[K^{\varphi}, G_0] = [K, G_0^{\varphi}]^{\varphi} \subseteq  Z(\nu(G_0))^{\varphi} =  Z(\nu(G_0))$.

Consider $k \in K, g_1, g_2 \in G_0$. We have  to show that (using left-normed commutators) $[k,
g_1^{\varphi}, g_2^{\varphi}] = 1 = [k,
g_1^{\varphi}, g_2]$. Note that by \cite[Lemma 2.1]{Norai} $$[k,
g_1^{\varphi}, g_2^{\varphi}] =  [k,
g_1^{\varphi}, g_2] = [k, g_1, g_2^{\varphi}] = [[k, g_1], g_2^{\varphi}] = 1$$
where the last equality follows from $[k, g_1] = 1$.

b) Note that
$$B = \langle K, K^{\varphi} \rangle N$$
where $N$ is the normal subgroup of $\nu(G_0)$ generated by $ [K, G_0^{\varphi}], [K^{\varphi}, G_0]$. By a) $N$ is generated as a group by  $[K, G_0^{\varphi}], [K^{\varphi}, G_0]$.

c) By a) and b) it remains only to show that $[K, K^{\varphi}] = 1$.
By \cite[Lemma 2.2]{Norai} if $x,a,b, \in G_0$ such that $[x,a] = 1 = [x,b]$ then $[x, [a,b]^{\varphi}] = 1$ in $\nu(G_0)$. This implies that $[K, [G_0,G_0]^{\varphi} ] = 1$. By construction $K \subseteq [G_0, G_0]$, hence $$[K, K^{\varphi}] = 1.$$

d) follows from a), b) together with $\nu(G_0) = \langle G_0, G_0^{\varphi} \rangle$ and $[K, G_0] = 1 = [K^{\varphi}, G_0^{\varphi}]$. This completes the proof of the Claim.

\medskip

As in the proof of Theorem \ref{nu} consider   $W_0(G_0)$ central in $\nu(G_0)$. 
 Hence $ W_0(G_0) B $ is abelian.

We have that $\nu(G_0)/ W_0(G_0) \simeq Im (\rho_0)$ a subgroup of $G_0 \times G_0 \times G_0$ and $$\nu(G_0) / B \simeq \nu (G_0/ K)$$ hence $$\nu(G_0)/ W_0(G_0) B \simeq \nu (G_0/ K) / W_0(G_0/ K)$$ is a subgroup of $(G_0/ K) \times (G_0/ K) \times (G_0/ K)$. Thus  $\nu(G_0)/ W_0(G_0) B$ is a metabelian group. 

By \cite{KochSidki} $G_0$ finitely presented implies that $\nu(G_0)$ is finitely presented. Since $G_0$ is soluble then $\nu(G_0)$ is soluble. Then  $\nu(G_0)/ W_0(G_0) B$ is a metabelian quotient of the finitely presented  soluble group $\nu(G_0)$, so $\nu(G_0)/ W_0(G_0) B$ is finitely presented, hence $W_0(G_0) B$ is finitely generated as a normal subgroup of $\nu(G_0)$.

Recall that  $W_0(G_0) B$ is abelian. By Claim d) $[B, \nu(G_0), \nu(G_0)] = 1$ that together with $W_0(G_0)$ central in $\nu(G_0)$ implies $$[W_0(G_0)B, \nu(G_0), \nu(G_0)] = 1$$ so $\nu(G_0)$  acts nilpotently  on $W_0(G_0) B$. Since $W_0(G_0) B$ is finitely generated as a normal subgroup of $\nu(G_0)$ we conclude that $$W_0(G_0) B \hbox{ is  a finitely generated abelian group}.$$ Then we can apply Theorem \ref{ab-pol} for $G = \nu(G_0)$ and  set 
$A = [\nu(G_0), \nu(G_0)] W_0(G_0) B$. We claim  that $\nu(G_0)$ is abelian-by-polycyclic. Indeed since $\nu(G_0)/ W_0(G_0) B$ is metabelian $$A' =[A, A] \subseteq  W_0(G_0) B,$$ hence $A'$ is   finitely generated abelian and $$[A', \nu(G_0), \nu(G_0)] \subseteq [W_0(G_0) B, \nu(G_0), \nu(G_0)] = 1$$
 Note that if $K = G_0''$ then $B$ is a characteristic subgroup of $\nu(G_0)$. By construction $W_0(G_0)$ is characteristic in $\nu(G_0)$, hence $A$ is a characteristic subgroup of $\nu(G_0)$. Then by Theorem \ref{ab-pol} $\nu(G_0)$ has an abelian characteristic subgroup $C$ such that $\nu(G_0)/ C$ is polycyclic.

Finally as $G_0 \otimes G_0$ is a subgroup of $\nu(G_0)$  we conclude that $\nu(G_0)$ and  $G_0 \otimes G_0$ are residually finite. Furthermore $G_0 \wedge G_0$ is a subgroup of $\nu(G_0)/ \Delta(G_0)$ and   since $\nu(G_0)/ \Delta(G_0)$ is a finitely generated abelian-by-polycyclic group, hence residually finite, we conclude that $G_0 \wedge G_0$ is residually finite.
\end{proof}

\section{Residual finiteness of the $q$-exterior square} \label{q-tensor123} 

Let $q \geq 1$ be an integer and $\mathbb{F}_q = \Z/ q \Z$. In  \cite{C-R}  the non-abelian  $q$-exterior square $G \wedge^q G$ and  the non-abelian q-tensor product $G \otimes^q G$  are defined, see \cite{Ellis2} too,
and by \cite{E-RF}  there is a  central extension \begin{equation} \label{ellis1}  1 \to H_2(G, \mathbb{F}_q) \to G \wedge^q G \to [G,G]G^q \to 1 \end{equation}
When $G$ is $q$-perfect i.e. $G = [G,G]G^q$ the above is the universal q-central extension studied in \cite{Br}.  There is a central extension $$1 \to \Delta^q(G) \to G \otimes^q G \to G \wedge^q G \to 1,$$ see \cite[Remark 2.2]{D-Rocco}.  We refer the reader to \cite{Irene-Norai} for more properties of $G \wedge^q G$.

Note that for $G \simeq F/ R$ where $F$ is a free group we have a central extension
$$ 1 \to \frac{R \cap [F,F] F^q }{[F,R] R^q} \to \frac{[F,F] F^q}{[F,R] R^q} \to \frac{[F,F]F^q}{R \cap [F,F] F^q} \to 1$$
with  $H_2(G, \mathbb{F}_q) \simeq  \frac{R \cap [F,F] F^q }{[F,R] R^q}$ and $\frac{[F,F] F^q}{R \cap [F,F] F^q} \simeq [G,G]G^q$. This implies $$G \wedge^q G \simeq \frac{[F, F] F^q}{[F,R] R^q}, $$  see \cite{McD}, \cite[Lemma 2.5]{D-Rocco}. Note that when $q = 1$ we have $H_2(G, \mathbb{F}_q) = 0$ and 
$G \wedge^q G \simeq [G,G]G^q = G$. Then we can assume without loss of generality that $q \geq 2$.

In this section we show that the methods introduced in the previous sections can be adapted to work for $G \wedge^q
G$. 

 For subsets $A,B, D$ of a profinite group $F_0$ write $\overline{[A, B]D^q}$ for the profinite subgroup of $F_0$ generated (topologically)  by $\{ [f_1, f_2], f_3^q \ | \ f_1 \in A, f_2 \in B, f_3 \in D \}$.

 \begin{lemma}
 Let $C$ be a profinite group and $F_0$ be a profinite group with a  normal  profinite subgroup $R_0$ such that $C \simeq F_0 / R_0$ and $0 = H_2(F_0, \mathbb{F}_q)$. Let $q \geq 2$ be an integer. Then there is an isomorphism
 $$H_2(C, \mathbb{F}_q) \simeq R_0 \cap \overline{[F_0, F_0]F_0^q}/ \overline{[R_0, F_0] R_0^q}. $$
 \end{lemma}
 
 \begin{proof}
 Consider the short exact sequence of profinite groups $1 \to R_0 \to F_0 \to C \to 1$
 and the associated Lyndon-Hoschild-Serre spectral sequence $E_{i,j}^2 = H_i(C, H_j(R_0, \mathbb{F}_q))$ that converges to the profinite group $H_{i+ j}(F_0, \mathbb{F}_q)$. This spectral sequence gives rise to a 5-term exact sequence in homology
 $$
 0 = H_2(F_0, \mathbb{F}_q) \to H_2(C, \mathbb{F}_q) \to R_0/\overline{[R_0, F_0] R_0^q} \to H_1(F_0, \mathbb{F}_q) \to  H_1(C, \mathbb{F}_q) \to 1$$
 Since $H_1(F_0, \mathbb{F}_q)  \simeq F_0/ \overline{[F_0, F_0]F_0^q}$ we obtain that
 $$ H_2(C,  \mathbb{F}_q) \simeq Ker( R_0/\overline{[R_0, F_0]R_0^q} \to H_1(F_0, \mathbb{F}_q)) \simeq$$ $$
 Ker ( R_0/ \overline{[R_0, F_0] R_0^q} \to  F_0/ \overline{[F_0, F_0]F_0^q}) = R_0 \cap \overline{[F_0, F_0]F_0^q}/ \overline{[R_0, F_0]R_0^q}.$$
 \end{proof}

Note that for the Whitehead quadratic functor we have that $\Gamma(A \oplus B) = \Gamma (A) \oplus \Gamma(B) \oplus A \otimes_{\Z} B$ where $A$ and $B$ are abelian groups.

\begin{lemma} \cite{Ellis2} \label{Wh-Wh}  Let $\Gamma$ be the Whitehead quadratic functor. Then there is an exact sequence
$$ \Gamma(G / [G,G] G^q) \to G \otimes^q G \to  G \wedge^q G \to 1$$
\end{lemma}

Recall that a group $G$ is $q$-perfect if $G = [G,G]G^q$. Then by Lemma \ref{Wh-Wh} 
 if $G$ is $q$-perfect then $G \otimes^q G \simeq  G \wedge^q G$.

 \begin{lemma} \label{res-finite25}
 Let $q \geq 2$ be an integer. Suppose that $G$ is a discrete group such that $G$ is residually finite and the canonical map $H_2(G, \mathbb{F}_q) \to H_2(\widehat{G}, \mathbb{F}_q)$ is injective. Then $G \wedge^q G$ is residually finite.
 \end{lemma}
 
 \begin{proof} The proof is a modification of the proof of Lemma \ref{res-finite}.
  Consider the short exact sequence
  $$ 1 \to H_2(G/U, \mathbb{F}_q) \to (G / U) \wedge^q (G/ U) \to [G/U, G/U](G/U)^q \to 1$$ 
  Then take
   inverse limit over all $G/ U$ finite and obtain an exact sequence
  $$ 1 \to \varprojlim_U  H_2(G/U, \mathbb{F}_q) \to \varprojlim_U (G / U) \wedge^q (G/ U) \to \varprojlim_U [G/U, G/U](G/U)^q $$
 Note that  $$\widehat{G} \widehat{\wedge}^q \widehat{G} : =  \varprojlim_U (G / U) \wedge^q (G/ U), $$   where the inverse limit is over all finite quotients $G/ U$, is a profinite group since $(G / U) \wedge^q (G/ U)$ is finite by (\ref{ellis1})    and $  \varprojlim_U [G/U, G/U](G/U)^q$ is the  closed subgroup $\overline{[\widehat{G}, \widehat{G}] \widehat{G}^q}$ of $ \widehat{G}$. 
  
  The canonical map $G \to \widehat{G}$ induces a map $\beta : G \wedge^q G \to \widehat{G} \widehat{\wedge}^q \widehat{G}$  that induces  a  commutative diagram with exact rows
 
\[ 
\begin{tikzcd}
1\arrow{r}{}&  H_2(G, \mathbb{F}_q)  \arrow{r}
\arrow{d}{\alpha}
 &   G \wedge^q G \arrow{r}{} \arrow{d}{\beta}& \ [G,G]G^q \arrow{r}{}\arrow{d}{\gamma}& 1\\%
1\arrow{r}{}& \varprojlim_U H_2( G/U, \mathbb{F}_q) \arrow{r}{}  &  \widehat{G}  \ \widehat{\wedge}^q \  \widehat{ G} \arrow{r}{} & \ \overline{[\widehat{G},\widehat{G}]\widehat{G}^q}\arrow{r}{}   & 1 \\%
\end{tikzcd}
\]
 where the exactness of the second row follows from the finiteness of all groups  $H_2( G/U, \mathbb{F}_q)$ and the Mittag-Leffler condition that implies that $ \varprojlim_U^1 H_2( G/U, \mathbb{F}_q)  = 0$.

  Note that $\varprojlim_U  H_2(G/U, \mathbb{F}_q) \simeq H_2(\widehat{G}, \mathbb{F}_q)$. By assumption the map $H_2(G, \mathbb{F}_q) \to  H_2(\widehat{G}, \mathbb{F}_q)$ is injective, hence $\alpha$ is injective.
  
  Since $G$ is residually finite the canonical map $G \to \widehat{G}$ is injective, hence $\gamma$ is injective. This implies, via the commutative diagram, that $\beta$ is injective but since $ \widehat{G} \widehat{\wedge}^q \widehat{G}$ is profinite, $ \widehat{G}   \widehat{\wedge}^q  \widehat{G}$  is residually finite, hence  $G \wedge^q G$ is residually finite.
  \end{proof}
\begin{cor}
 Suppose that $G$ is a residually finite group such that $G \in \mathcal{A}_2$ and $q \geq 2$ is an integer.
Then $G \wedge^q G$ is residually finite.
\end{cor}

\begin{proof}  $G \in \mathcal{A}_2 = \mathcal{A}_2^{hom}$ implies that the map $H_2(G, \mathbb{F}_q) \to H_2(\widehat{G}, \mathbb{F}_q)$ is an isomorphism and we can apply Lemma \ref{res-finite25}.
\end{proof}
  
    \begin{prop} \label{iff25} Let $q \geq 2$ be an integer.
  Suppose that $G$ is a finitely generated residually finite group. Then $G \wedge^q G$ is residually finite if and only if the map $H_2(G, \mathbb{F}_q) \to H_2(\widehat{G}, \mathbb{F}_q)$, induced by the canonical map  $G \to \widehat{G}$, is injective.
  \end{prop}

\begin{proof}  Let $F$ be a finitely generated free group with a normal subgroup $R$ such that  $F/ R \simeq G$. Suppose that $G \wedge^q G$ is residually finite. Consider the short exact sequence of discrete groups
 \begin{equation} \label{129}  1 \to \frac{[F,F]F^q}{[R,F]R^q} \to \frac{F}{[R,F]R^q} \to \frac{F}{[F,F]F^q} \to 1 \end{equation}
Recall that  $$\frac{[F,F]F^q}{[R,F]R^q}  \simeq G \wedge^q G$$
Note that if $$1 \to N_0 \to E_0 \to G_0 \to 1$$ is a short exact sequence of groups  where $N_0$  is finitely generated, residually finite and $G_0$  is finite, then $E_0$ is residually finite. We apply this for the short exact sequence (\ref{129}) and conclude that $\frac{F}{[R,F]R^q}$ is residually finite and so the canonical map $$\beta: \frac{F}{[R,F]R^q}   \to \frac{ \widehat{F}}{\overline{[\overline{R},\widehat{F}]{\overline{R}}^q}} \hbox{ is injective,}$$ since  $\frac{ \widehat{F}}{\overline{[\overline{R},\widehat{F}]{\overline{R}}^q}}$ is the profinite completion of $\frac{F}{[R,F]R^q} $.
 Since $$H_2(G, \mathbb{F}_q) \simeq \frac{R \cap ([F,F] F^q)} {[R, F] R^q} \hbox{
  and }H_2( \widehat{G}, \mathbb{F}_q) \simeq  \frac{ \overline{R} \cap (\overline{[\widehat{F},\widehat{F}] \widehat{F}^q})} {\overline{[\overline{R}, \widehat{F}] \overline{R}^q}}$$ we can identify the restriction of $\beta$ to $\frac{R \cap ([F,F] F^q)} {[R, F] R^q}$ as  the canonical map $H_2(G, \mathbb{F}_q) \to H_2( \widehat{G}, \mathbb{F}_q )$. As $\beta$ is injective  $H_2(G, \mathbb{F}_q) \to H_2( \widehat{G}, \mathbb{F}_q )$ is injective.
  
  The converse is Lemma \ref{res-finite25}.
\end{proof}

Recall that $ \widehat{G}_p$ is the pro-$p$ completion of $G$. 

\begin{cor} Suppose that $G$ is a  residually $p$-finite group,  $p,q \geq 2$ are integers such that $p$ is prime and $q \geq 2$. Then 

 a) if the map $H_2(G, \mathbb{F}_q) \to H_2(\widehat{G}_p, \mathbb{F}_q) $, induced by the canonical map  $G \to \widehat{G}$, is injective then $ G \wedge^q G$ is residually $p$-finite,

b) if $G$ is finitely generated and $ G \wedge^q G$ is residually $p$-finite then  the map $H_2(G, \mathbb{F}_q) \to H_2(\widehat{G}_p, \mathbb{F}_q) $, induced by the canonical map  $G \to \widehat{G}_p$, is injective. 
\end{cor}

\begin{proof} In the proof of Lemma \ref{res-finite25}, Proposition \ref{iff25}  we substitute finite quotients with $p$-finite quotients, profinite completions $\widehat{G}$  and $ \widehat{F}$ with pro-$p$ completions $\widehat{G}_p$ and $\widehat{F}_p$, residually finite with residually $p$-finite. 
\end{proof}

  \begin{prop} \label{agosto28}  Let $G$ be a finitely generated group, $F$ a finitely generated free group such that $G \simeq F/ R$. If
$F/[F,R]R^q$ is abelian-by-polycyclic, then $H_2(G, \mathbb{F}_q)$ is finite and $G$ and $G \wedge^q G$ are residually finite. 
\end{prop}

\begin{proof} In the proof of Proposition \ref{ab-pol1} was shown that  the condition that $F/[F,R]$ is abelian-by-polycyclic  implies that $R/[F,R]$ is a finitely generated abelian group. 
If we repeat the same argument but start with $F/[F,R]R^q$ is abelian-by-polycyclic   we can conclude  that $R/[F,R]R^q$ is a finitely generated abelian group. Since $R/[F,R]R^q$ has  finite exponent it is finite. Then  $$H_2(G, \mathbb{F}_q) \simeq R \cap [F,F]F^q/ [F,R]R^q$$ is a subgroup of $R/[F,R]R^q$, hence is finite. 

Since $F/[F,R]R^q$  is finitely generated abelian-by-policyclic, it  is residually finite.
Since $G \wedge^q G$ is a subgroup of $F/[F,R]R^q$ then  $G \wedge^q G$ is residually finite too. Finally $G \simeq F/ R$ is a quotient of $F/[F,R]R^q$, hence $G \simeq F/ R$ is finitely generated, abelian-by-polycyclic and residually finite.

\end{proof}

 \begin{cor} \label{fin-pres2}
 Let $G$ be a finitely presented metabelian group and $q \geq 2$ be an integer. Then
 
 a) $F/ [F,R]R^q$ is abelian-by-polycyclic,
 
 b) the map $H_2(G, \mathbb{F}_q) \to H_2(\widehat{G}, \mathbb{F}_q)$ is injective,

 c)  $G \wedge^q G$ is residually finite. 
 \end{cor}

\begin{proof} Let $G \simeq F / R$ where $F$ is a finitely generated free group.
By the proof of Theorem \ref{fin-pres1} $F/ [F,R]$ is abelian-by-polycyclic, hence  $F/ [F,R]R^q$ is abelian-by-polycyclic.
Then we can apply Proposition \ref{iff25} and Proposition \ref{agosto28}.
\end{proof}

 Recall that  in \cite{Ellis2} a generalisation of $\nu(G)$  related to $G \otimes^q G$ and $G \wedge^q G$ was defined. It was later denoted by $\nu^q(G)$ in \cite{Bu-Ro} and further properties of $\nu^q(G)$ were investigated in \cite{Bu-Ro}.

\begin{prop} Let $ q \geq 1$ be an integer.
Suppose that $\nu(G)$ is a finitely generated abelian-by-polycyclic group. Then $\nu^q(G)$ is residually finite. In particular if $G$ is finitely presented, central-by-metabelian then $\nu^q(G)$ is residually finite.
\end{prop}

\begin{proof} By \cite[Prop. 2.6]{Bu-Ro} there is an exact sequence $$\nu(G) \to \nu^q(G) \to G / G' \to 1$$ Let $S$ be the image of $\nu(G)$ in $\nu^q(G)$. Since $S$ is finitely generated, abelian-by-polycyclic it is residually finite. The class $\mathcal{A}_2$ is extension closed, hence every finitely generated abelian group is in $\mathcal{A}_2$, in particular $G/ G' \in \mathcal{A}_2$. Then by Proposition \ref{SerSer}  applied 
for the short exact sequence of groups $$1 \to S \to \nu^q(G) \to G/ G' \to 1$$ we conclude that $\nu^q(G)$ is residually finite. \end{proof}

In the case when $G$ is finitely presented, centre-by-metabelian the above proposition is   strengthened in the following result.

\begin{theorem} (= Theorem F) Let $ q \geq 1$ be an integer.
Suppose that  $G$ is finitely presented, central-by-metabelian group. Then there is an abelian characteristic subgroup $C_0$ of $\nu^q(G)$  such that $\nu^q(G)/ C_0$ is polycyclic.  In particular, $\nu^q(G)$, $G \otimes^q G$ and $G \wedge^q G$ are residually finite.
\end{theorem}
\begin{proof}
By Corollary \ref{ThmDThm} there is a characteristic subgroup $C$ of $\nu(G)$ such that $\nu(G)/ C$ is polycyclic. Since the map $\nu(G) \to \nu^q(G)$ is functorial on $G$, the image $S$ of $\nu(G)$ in $\nu^q(G)$ is characteristic in $\nu^q(G)$ and the image $C_0$ of $C$ in $\nu^q(G)$ is characteristic in $S$, hence is characteristic in $\nu^q(G)$. In particular $C_0$ is an abelian normal subgroup of $\nu^q(G)$. Finally there is a short exact sequence of groups $$ 1 \to S/ C_0 \to \nu^q(G)/ C_0 \to G/ G' \to 1$$ where $G/ G'$ is finitely generated abelian and $S/ C_0 \simeq \nu(G) / C$ is polycyclic. Hence $\nu^q(G)/ C_0$ is polycyclic. Using a similar argument to the one at the end of Corollary  \ref{ThmDThm}, we conclude that $\nu^q(G)$, $G \otimes^q G$ and $G \wedge^q G$ are abelian-by-polycyclic, thus residually finite.
\end{proof}

\end{document}